 \documentclass[11pt]{article}
 
\usepackage{amssymb,amsmath,amsfonts}
\usepackage{graphicx,color,enumitem}
\usepackage{mathrsfs}
\usepackage{amsthm} 
\usepackage[dvipsnames]{xcolor}
\usepackage{bm}
\usepackage[round]{natbib}
\usepackage[]{appendix}

\usepackage{geometry}

\usepackage[colorinlistoftodos, textwidth=4cm, shadow]{todonotes}
\RequirePackage[colorlinks,citecolor=blue,urlcolor=blue]{hyperref}

\usepackage{tikz}
\tikzstyle{vertex}=[circle, draw, inner sep=2pt, fill=white]

\newcommand{\s}{{\mathcal S}}
\newcommand{\e}{{\varepsilon}}

\newcommand{\E}{{\mathbb E}}

\renewcommand{\P}{{\mathbb P}}

\newcommand{\R}{{\mathbb R}}

\newcommand{\N}{{\mathbb N}}

\newcommand{\Acal}{{\mathcal A}}

\newcommand{\Fcal}{{\mathcal F}}
\newcommand{\Gcal}{{\mathcal G}}

\newcommand{\Scal}{{\mathcal S}}

\newcommand{\Xcal}{{\mathcal X}}

\newcommand{\tg}{{g}}

\newcommand{\cC}{{\alpha}}

\newcommand{\fdot}{{\,\cdot\,}}

\newcommand{\id}{{\mathrm{id}}}

\DeclareMathOperator{\tr}{Tr}

\DeclareMathOperator{\supp}{supp}

\newtheorem{theorem}{Theorem}

\newtheorem{corollary}[theorem]{Corollary}

\newtheorem{definition}[theorem]{Definition}
\newtheorem{remark}[theorem]{Remark}
\newtheorem{example}[theorem]{Example}

\newtheorem{lemma}[theorem]{Lemma}

\newtheorem{proposition}[theorem]{Proposition}

\numberwithin{equation}{section}
\numberwithin{theorem}{section}

\begin{document}

\title{Probability measure-valued polynomial diffusions}
\author{Christa Cuchiero\thanks{Faculty of Mathematics, Vienna University, Oskar-Morgenstern-Platz 1, A-1090 Wien, Austria, christa.cuchiero@univie.ac.at}
\and Martin Larsson\thanks{Department of Mathematics, ETH Zurich, R\"amistrasse 101, CH-8092, Zurich, Switzerland, martin.larsson@math.ethz.ch.}
\and Sara Svaluto-Ferro\thanks{Faculty of Mathematics, Vienna University, Oskar-Morgenstern-Platz 1, A-1090 Wien, Austria, sara.svaluto-ferro@univie.ac.at.\newline
Christa Cuchiero and Sara Svaluto-Ferro gratefully acknowledge financial support by the Vienna Science and Technology Fund
(WWTF) under grant MA16-021.
 Martin Larsson and Sara Svaluto-Ferro gratefully acknowledge financial
support by the Swiss National Science Foundation (SNF) under grant 205121$\_$163425.
}}

\maketitle

\begin{abstract}

We introduce a class of probability measure-valued diffusions, coined \emph{polynomial}, of which the well-known Fleming--Viot process is a particular example. The defining property of finite dimensional polynomial processes considered by \cite{CKT:12, FL:16} is transferred to this infinite dimensional setting. This leads to a representation of conditional marginal moments via a finite dimensional linear PDE, whose spatial dimension corresponds to the degree of the  moment. As a result, the tractability of finite dimensional polynomial processes are preserved in this setting.  We also obtain a representation of the corresponding extended generators, and prove well-posedness of the associated martingale problems. In particular, uniqueness is obtained from the duality relationship with the PDEs mentioned above. 

\end{abstract}

\tableofcontents

\section{Introduction}

In this paper we develop probability measure-valued versions of a class of processes known as polynomial diffusions, which have -- due to their inherent tractability -- broad applications in population genetics, interacting particle systems, and finance; see e.g.~\cite{E:11, V:90, FK:05}. The result is a class of stochastic processes that model randomly evolving probability measures, including examples such as Fleming--Viot processes \citep{FV:79, EK:93}, as well as conditional laws of jump-diffusions on (subsets of) $\R^d$.

Finite dimensional polynomial diffusions form a rich class that includes Kimura diffusions \citep{K:64}, Wishart correlation matrices \citep{AA:13}, and affine processes \citep{DFS:03}, just to name a few subclasses. See e.g.\ \cite{CKT:12, FL:16, FL:17, CLS:17} for further details and examples. This suggests transferring their defining property and tractability features to infinite dimensional processes. Such  processes also appear as limits of empirically  well-suited finite-dimensional polynomial models, whose limiting behavior is of key interest in population dynamics, but also in other areas, such as capital distribution curve modeling; see e.g.\ \cite{S:13}.

The infinite dimensional setup that we consider here are polynomial diffusions $X$  taking values in the space of probability measures on a locally compact Polish space $E$. We define them as path-continuous solutions of martingale problems for certain operators $L$ acting on classes of cylinder polynomials, i.e.~functions $p$ of the form
\[
\textstyle p(\nu)=\phi\left(\int_E g_1(x)\nu(dx),\ldots,\int_E g_m(x)\nu(dx)\right),
\]
where $\phi$ is a polynomial in $m$ variables, $g_1,\ldots,g_m$ are continuous and bounded, and the argument $\nu$ is a probability measure. Any such function $p$ can be regarded as a homogeneous polynomial in the probability measure $\nu$, and admits a natural notion of degree as discussed in Section~\ref{IIsec:2}. The defining property of a probability measure-valued polynomial diffusion is that $Lp$ is again a homogeneous polynomial with the same degree as $p$ (or is the zero polynomial). The precise definitions are actually somewhat more general; the details are in Section~\ref{sec_pol_op}.

A consequence is that moments admit tractable representations. Specifically, for a probability measure-valued polynomial diffusion $X$ starting at $X_0=\nu$, we establish in Section~\ref{sec_moments_uniqueness} (under suitable conditions) the moment formula
\begin{equation} \label{intro_ml_2}
\E\Big[\int_{E^k} g(x_1,\ldots,x_k) X_t(dx_1)\cdots X_t(dx_k)\Big] = \int_{E^k} u(t,x_1,\ldots,x_k) \nu(dx_1)\cdots \nu(dx_k),
\end{equation}
where $u(t,x_1,\ldots,x_k)$ solves the linear partial integro-differential equation (PIDE)
\begin{equation}\label{PIDEintro}
\frac{\partial u}{\partial t} =  \overline{L}_k u \quad\text{in $(0,\infty)\times E^k$}
\end{equation}
with initial data $u(0,x_1,\ldots,x_k)=g(x_1,\ldots,x_k)$, and where $\overline L_k$ is a linear operator derived from the generator $L$ of $X$, acting on (a subspace of) $C(E^k)$. The $k$-dimensional PIDE \eqref{PIDEintro} is significantly simpler than the Kolmogorov equation, whose state space in this context consists of measures on $E$. Indeed, \eqref{PIDEintro} corresponds to the Feynman-Kac PIDE associated to an  $E^k$-valued Markov process. When $E$ consists of finitely many points, we recover the finite dimensional case where \eqref{PIDEintro} reduces to a linear ODE associated to a certain Markov chain with values in $E^k$, whose solution is computed by matrix exponentiation.  These PIDEs fall exactly in the setup considered by \cite{BBGJJ:18}, who develop numerical solution procedures based on neural networks. These methods do not suffer from the curse of dimensionality, and give the whole function $(x_1,\ldots,x_k) \mapsto u(t, x_1,\ldots,x_k)$.
 
The moment formula forms a particular instance of {\em duality}, which is often used to prove uniqueness for measure-valued martingale problems. This is the case also here, and we obtain uniqueness under broad circumstances. Being solutions of PIDEs, our dual ``processes'' are deterministic, in contrast to other commonly used duals such as the Kingman coalescent in the Fleming--Viot case; see e.g.\ \cite{DH:82}. Note that moment formulas for Fleming--Viot type process are classical (see e.g.\ \cite{DH:82} or \cite[Section 2.8]{D:93}); we show here that they are actually available much more broadly.

Existence of measure-valued processes is often proved via large population limits of carefully constructed particle systems; see e.g.\ \cite[Section~2]{D:93} for a semigroup approach, or \cite{EK:93, EK:87} for an approach via martingale problems. We also work with martingale problems, but rather than using approximations by finite particle systems, we obtain existence directly via the positive maximum principle. This relies on new optimality conditions for polynomials of measure arguments developed in Section~\ref{IIstool}. As a result, we can describe large parametric families of specifications. In particular, we obtain a full characterization of probability measure-valued polynomial diffusions whose generator $L$ has a sufficiently large domain. This yields extensions of the so-called \emph{Fleming--Viot process with weighted sampling} discussed in \cite[Section 5.7.8]{D:93}, where the sampling-replacement rate is allowed to depend on the type. On the other hand, by restricting the domain of $L$ we obtain a richer class, including e.g.\ the model of exchangeable diffusions considered by \cite{V:88}; see also \cite[Section 5.8.1]{D:93}. Our existence and well-posedness results are in Section~\ref{IIs40}.

Arguably, the bulk of applications of measure-valued processes come from population genetics. But tractable specifications like those developed here are for instance also of interest in non-parametric Bayesian statistics (see e.g.~\cite{RGN:02, RLP:03} who consider distributions of functionals of random probability measures), age distribution and longevity risk modeling (see e.g.~\cite{BHEA:18}), or high-dimensional financial modeling.
Let us sketch a situation from stochastic portfolio theory (see \cite{F:02, FK:09} for an introduction to this subject.) Let $Z$ be a process with values in the unit simplex $\Delta^d=\{z\in[0,1]^d\colon z_1+\ldots+z_d=1\}$, representing the capitalization weights of $d$ stocks. For tractability, it is natural select $Z$ to be a polynomial diffusion on $\Delta^d$ as in \cite{C:16}. To compute basic moment statistics of the capitalization weights, one uses a moment formula similar to \eqref{intro_ml_2}. For homogeneous polynomials $q(z_1,\ldots,z_d)$, it takes the form
\[
\E[q(Z_t)] = \sum_{\bm\alpha} u_t(\bm\alpha)\, z_1^{\alpha_1}\cdots z_d^{\alpha_d} ,
\]
where $Z_0=z\in\Delta^d$, and the sum extends over all multi-indices $\bm\alpha=(\alpha_1,\ldots,\alpha_d)$ with $|\bm\alpha|=\alpha_1+\cdots+\alpha_d=k:=\deg(q)$. There are $N:=\binom{k+d-1}{k}$ such multi-indices, and the $\R^N$-valued function $u_t=(u_t(\bm\alpha)\colon |\bm\alpha|=k)$ solves the linear ODE
\begin{equation} \label{intro_ml_1}
\frac{\partial u}{\partial t} = L_k u,
\end{equation}
whose initial condition is the coefficient vector of $q$, and where $L_k$ here is an $N\times N$ matrix derived from the generator of $Z$. For small or moderate dimensions $d$ and degrees $k$, solving \eqref{intro_ml_1} is feasible. However, $d$ is typically on the order of $10^3$, which renders \eqref{intro_ml_1} computationally taxing even for small $k$, since the ODE dimension is $N\sim d^k$.

Now, consider instead a \emph{linear factor model} $\widetilde Z=(\widetilde Z^1,\ldots,\widetilde Z^d)$ for the capitalization weights. This means that $\widetilde Z^i=\int_E g_i(x)X_t(dx)$ for some nonnegative functions $g_1,\ldots,g_d$ that sum to one, and a probability measure-valued polynomial diffusion $X$ with, say, $E=[0,1]$. In this case,
\[
\E[q(\widetilde Z_t)] = \E[p(X_t)]
\]
for some measure polynomial $p(\nu)$ of degree $k=\deg(q)$. This expectation can be computed using the moment formula \eqref{intro_ml_2}, which amounts to solving the PDE \eqref{PIDEintro} up to time $t$. Discretizing the space domain $E^k$ using $n$ points in each dimension yields a complexity of order $n^k$. This can be made orders of magnitude smaller than the complexity $d^k$ of solving \eqref{intro_ml_1}. Importantly, $n$ is a parameter that is chosen based on accuracy requirements, while $d$ is an input to the problem. This illustrates how probability measure-valued polynomial diffusions can enhance tractability in high-dimensional models. On top of this, as projections of an infinite-dimensional process, these linear factor models constitute a much richer class than polynomial models on subsets of $\Delta^d$.

The remainder of the paper is organized as follows. After reviewing some basic notation and definitions in the following subsection, we turn to polynomials of measure arguments in Section~\ref{IIsec:2}, and prove optimality conditions for such polynomials in Section~\ref{IIstool}. In Section~\ref{sec_pol_op} we define polynomial operators and study their form in the diffusion case. Section~\ref{sec_ex_un} contains the moment formula as well as our main results on well-posedness of the martingale problem. Applications and examples are treated in Section~\ref{IIsec:examples}. Some proofs and supplementary material are gathered in appendices.

\subsection{Notation and basic definitions}\label{IIs1}

Throughout this paper, $E$ is a locally compact Polish space endowed with its Borel $\sigma$-algebra.  The following notation is used.

\begin{itemize}
\item $M_+(E)$ denotes the finite measures on $E$, $M_1(E)\subset M_+(E)$ the probability measures, and $M(E)=M_+(E)-M_+(E)$ the signed measures of bounded variation (i.e., of the form $\nu_+-\nu_-$ with $\nu_+,\nu_-\in M_+(E)$). All three are topologized by weak convergence, which turns $M_+(E)$ and $M_1(E)$ into Polish spaces. For $\mu,\nu\in M(E)$ we write $\mu\leq\nu$ if $\nu-\mu\in M_+(E)$ and $|\nu|$ for $\nu_++\nu_-$.

\item $C(E)$, $C_b(E)$, $C_0(E)$, $C_c(E)$ have the usual meaning of continuous (and bounded, and vanishing at infinity, and compactly supported) real functions on $E$. The topology on the latter three is that of uniform convergence, and $\|\fdot\|$ denotes the supremum norm.

\item If $E$ is noncompact, then $E^\Delta=E\cup\{\Delta\}$ is the one-point compactification, itself a compact Polish space. If $E$ is compact we write $E^\Delta=E$, which mitigates the need to consider the compact and noncompact cases separately. We also define
$$C_\Delta(E^k):=\big\{f|_{E^k}\ :\ f\in C((E^\Delta)^k)\big\},$$
a closed subspace of $C_b(E^k)$. The spaces $C_\Delta(E)$ and $C(E^\Delta)$ can be identified, and we occasionally regard elements of the former as elements of the latter, and vice versa. When $E$ is compact, we have $C(E)=C_b(E)=C_0(E)=C_c(E)=C_\Delta(E)$ and we then simply write $C(E)$. Note that the constant function $1$ lies in $C_\Delta(E)$, but of course not in $C_0(E)$. This is one reason the spaces $C_\Delta(E^k)$ are useful; other reasons are discussed in Remarks~\ref{IIrem3} and~\ref{IIrem4}.

\item $\widehat C_\Delta (E^k)$ is the closed subspace of $C_\Delta (E^k)$ consisting of symmetric functions $f$, i.e., 
$f(x_1,\ldots,x_k)=f(x_{\sigma(1)},\ldots,x_{\sigma(k)})$ for all $\sigma\in \Sigma_k$, the permutation group on $k$ elements.
$\widehat C_0(E^k)$ and $\widehat C(E^k)$ are defined similarly.
 For any $g\in\widehat C_\Delta (E^k), h\in \widehat C_\Delta (E^\ell)$ we denote by
$g\otimes h\in\widehat C_\Delta(E^{k+\ell})$ the symmetric tensor product, given by
\begin{equation} \label{IIeqn17}
(g\otimes h) (x_1,\ldots,x_{k+\ell}) = \frac{1}{(k+\ell)!} \sum_{\sigma\in\Sigma_{k+\ell}} g\big(x_{\sigma(1)},\ldots,x_{\sigma(k)}\big)h\big(x_{\sigma(k+1)},\ldots,x_{\sigma(k+\ell)}\big).
\end{equation}
For a linear subspace $D\subseteq C_\Delta(E)$ we set $D\otimes D:=\text{span}\{g\otimes g\ :\ g\in D\}$. We emphasize that only symmetric tensor products are used in this paper.
\end{itemize}

Two key notions are the \emph{positive maximum principle} and \emph{conservativity} for certain linear operators. In general, for a Polish space $\Xcal$ and a subset $\s\subseteq \Xcal$, these notions are defined as follows.  An operator $\Acal\colon D\to C_b (\Xcal)$ with domain $D\subseteq C_b (\Xcal)$ is said to satisfy the positive maximum principle on $\s$ if 
$$\text{$f\in D$, $x\in \s$, $\sup_\s f=f(x)\geq0\quad$ implies $\quad\Acal f(x)\leq0$.}$$
If $\Scal$ locally compact, $\Acal$ is called $\Scal$-conservative if there exist functions $f_n\in D\cap C_0(\Scal)$ such that  $\lim_{n\to\infty}f_n=1$ on $E$ and $\lim_{n\to\infty}(\Acal f_n)^-=0$  on $E^\Delta$, both in the bounded pointwise sense; c.f.~Chapter 4.2 in \cite{EK:05}. For us, $\Scal$ will be $E$, $E^\Delta$, $M_1(E)$, or $M_1(E^\Delta)$.

It is well-known that the positive maximum principle, combined with conservativity, is essentially equivalent to the existence of a $\Scal$-valued solutions to the martingale problem for $\Acal$; see for instance Theorem 4.5.4 of \cite{EK:05}.  We use this extensively, and review the relevant results in Section~\ref{app_existence}. Here an important issue is that while $M_1(E)$ is compact when $E$ is compact, $M_1(E)$ is not even locally compact when $E$ is noncompact.

\section{Polynomials of measure arguments} \label{IIsec:2}

In this section we develop some basic properties of polynomials of measure arguments. The notation and results introduced here play a central role throughout this paper. Throughout this section $E$ is a locally compact Polish space.

\subsection{Monomials and polynomials}\label{IIs21}

A \emph{monomial} on $M(E)$ is an expression of the form 
\[
\langle g, \nu^k \rangle = \int_{E^k} g(x_1,\ldots,x_k) \nu(dx_1) \cdots \nu(dx_k)
\]
for some $k\in\N_0$, where $g\in \widehat C_\Delta (E^k)$ is referred to as the \emph{coefficient} of the monomial; see e.g.~\cite[Chapter 2]{D:93}. We identify $\widehat C_\Delta (E^0)$ with $\R$, so that for $k=0$ we have $\langle g,\nu^0\rangle=g\in\R$. It is clear that the map $\nu\mapsto\langle g,\nu^k\rangle$ is homogeneous of degree~$k$, and that $g\mapsto\langle g,\nu^k\rangle$ is linear. Furthermore, one has the identity $\langle g,\nu^k\rangle \langle h,\nu^\ell\rangle = \langle g\otimes h,\nu^{k+\ell}\rangle$, where the symmetric tensor product $g\otimes h$ is defined in \eqref{IIeqn17}.

A polynomial on $M(E)$ is now defined as a (finite) linear combination of monomials,
\begin{equation} \label{IIeq:p(nu)}
p(\nu) = \sum_{k=0}^m \langle g_k,\nu^k\rangle,
\end{equation}
with coefficients $g_k\in\widehat C_\Delta (E^k)$. The degree of the polynomial $p(\nu)$, denoted by $\deg(p)$, is the largest $k$ such that $g_k$ is not the zero function, and $-\infty$ if $p$ is the zero polynomial. The representation~\eqref{IIeq:p(nu)} is unique; see Corollary~\ref{IIC:pol uniqueness} below.

\begin{example}\label{IIex4}
Let $E=\{1,\ldots,d\}$ be a finite set. Then every element $\nu\in M(E)$ is of the form
\[
\nu = z_1\delta_1 + \cdots + z_d \delta_d, \qquad (z_1,\ldots,z_d)\in\R^d,
\]
where $\delta_i$ is the Dirac mass concentrated at $\{i\}$. Monomials take the form
\[
\langle g,\nu^k\rangle = \sum_{i_1,\ldots,i_k} g(i_1,\ldots,i_k)\, z_{i_1}\cdots z_{i_k},
\]
where the summation ranges over $E^k=\{1,\ldots,d\}^k$. Therefore, as $g(\cdot)$ ranges over all symmetric functions on $E^k$, we recover all homogeneous polynomials of total degree $k$ in the $d$ variables $z_1,\ldots,z_d$. In particular,  in view of Corollary \ref{IIrem5} later, this relation provides a one to one correspondence between polynomials on the unit simplex  $\Delta^d$, namely
$$\Delta^d:=\Big\{z\in\R^d\colon\sum_{i=1}^dz_i=1,\ z_i\geq0\Big\},$$
and polynomials on $M_1(E)$. 
\end{example}

The following function space will play an important role.

\begin{definition} \label{IID:Pol(M(E))}
Let
\[
P := \left\{\nu\mapsto p(\nu) \colon \text{$p$ is a polynomial on $M(E)$} \right\}
\]
denote the algebra of all polynomials on $M(E)$ regarded as real-valued maps, equipped with pointwise addition and multiplication.
\end{definition}

\subsection{Continuity and smoothness of polynomials}

Just like ordinary polynomials, the elements of $P$ are smooth. This is made precise in Lemma~\ref{IIL:Psmooth} below. In its statement, we use a directional derivative of functions on $M(E)$ that is well-known since the work of \cite{FV:79}.
A function $f\colon M(E)\to\R$ is called differentiable at $\nu$ in direction $\delta_x$ for $x\in E$ if
\[
\partial_x f(\nu) := \lim_{\varepsilon\to0} \frac{f(\nu+\varepsilon\delta_x)-f(\nu)}{\varepsilon}
\]
exists. We write $\partial p(\mu)$ for the map $x\mapsto\partial_xp(\mu)$, and use the notation
\[
\partial^k_{x_1x_2\cdots x_k} f(\nu) := \partial_{x_1}\partial_{x_2}\cdots \partial_{x_k} f(\nu)
\]
for iterated derivatives. We write $\partial^k p(\nu)$ for the corresponding map from $E^k$ to $\R$. Observe that for $p\in P$ of the form $p(\nu)=\langle g, \nu\rangle$ we get
$\partial_x p(\nu)=\lim_{\e\to0}(\int g(y) \e \delta_x(dy))\e^{-1}=g(x)$ for each $x\in E$.

The following lemma asserts basic properties of polynomials, in particular that 
polynomials on $M(E)$ can be uniquely extended to polynomials on $M(E^{\Delta})$, which will often be the object of interest for our purposes.

\begin{lemma}  \phantomsection \label{IIL:Psmooth}
\begin{enumerate}

\item\label{IIL:Psmooth:c1}  Each $p\in P$ is continuous on $M_+(E)$, sequentially continuous on $M(E)$, and can be uniquely extended to a polynomial on $M(E^\Delta)$.\footnote{It can be shown that sequential continuity cannot be strengthened to continuity.}
\item\label{IIL:Psmooth:1} Let $p\in P$ be a monomial of the form $p(\nu)=\langle g,\nu^k\rangle$. Then, for every $x\in E$ and $\nu\in M(E)$,
\[
\partial_x p(\nu) = k\langle  g(\fdot,x), \nu^{k-1}\rangle,
\]
where $g(\fdot,x) \in \widehat C_\Delta (E^{k-1})$ is the function $(x_1,\ldots,x_{k-1})\mapsto g(x_1,\ldots,x_{k-1},x)$. If $k=0$, the right-hand side should be read as zero.
\item\label{IIL:Psmooth:c2} For each $p\in P$ and $x\in E$ the map $\partial_x p\colon \nu\mapsto \partial_xp(\nu)$ lies in $P$.
\item\label{IIL:Psmooth:2} For each $p\in P$ and $\nu\in M(E)$, the map $\partial p(\nu)\colon x\mapsto \partial_xp(\nu)$ lies in $C_\Delta (E)$.
\item\label{IIL:Psmooth:3} The identity
\[
\partial_x(pq)(\nu)=p(\nu)\partial_xq(\nu)+q(\nu)\partial_xp(\nu)
\]
holds for all $p,q\in P$, $x\in E$, $\nu\in M(E)$.
\item\label{IIL:Psmooth:4} The Taylor representation
$$
p(\nu+\mu)=\sum_{\ell=0}^{k}\frac 1 {\ell!}\langle\partial^\ell p(\nu),\mu^\ell\rangle,
$$
holds for all $p\in P$ and $\nu,\mu\in M(E)$, where $k$ denotes the degree of $p$.

\end{enumerate}
\end{lemma}

\begin{proof}
\ref{IIL:Psmooth:c1}: For $h\in C_\Delta(E)^{\otimes k}$ we can write $h=\sum_{\ell=1}^L\lambda_\ell h_\ell^{\otimes k}$ for some $h_\ell\in C_\Delta(E)$ and $\lambda_\ell\in\R$. Since $\langle h_\ell,\nu\rangle$ is continuous by definition of weak convergence, 
$$\langle h,\nu^k\rangle=\sum_{\ell=1}^L\lambda_\ell\langle h_\ell^{\otimes k},\nu^k\rangle=\sum_{\ell=1}^L\lambda_\ell\langle h_\ell,\nu\rangle^k$$
 is continuous as well.
Note then that by linearity in \eqref{IIeq:p(nu)} it is enough to prove the result for $p(\nu)=\langle g,\nu^k\rangle$ and $g\in \widehat C_\Delta(E^k)$. Choose $h\in C_\Delta(E)^{\otimes k}$ such that $\|g-h\|\leq\e$ and  let $\nu_n\in M(E)$ 
form a convergent sequence with limit $\nu\in M(E)$. Observe that, by the Banach--Steinhaus theorem, $\sup_n|\nu_n|(E)<\infty$. Then
$$\big|\langle g,\nu_n^k\rangle-\langle g,\nu^k\rangle\big|\leq \big|\langle h,\nu_n^k\rangle-\langle h,\nu^k\rangle\big|+\e\big(\sup_n|\nu_n|(E)^k+|\nu|(E)^k\big)\to C \e$$
for some $C\geq0$. Since $\e$ is arbitrary, this proves sequential continuity of $p$ on $M(E)$. In particular we get continuity on $M_+(E)$ since this is a Polish space. The last part follows from the observation that every function in $C_\Delta(E)$ can be uniquely extended to a function in $C(E^\Delta)$.

\ref{IIL:Psmooth:1}: Using  the symmetry of $g$, a direct calculation yields
$$
p(\nu+\varepsilon\delta_x)-p(\nu) = \varepsilon k \int g(x_1,\ldots,x_{k-1},x) \prod_{j=1}^{k-1}\nu(dx_j) + o(\varepsilon).
$$
The expression for $\partial_xp(\nu)$ follows.

For the remaining part of the proof it suffices to consider monomials $p(\nu)=\langle g,\nu^k\rangle$ for $g\in \widehat C_\Delta(E^k)$ due to the linearity in \eqref{IIeq:p(nu)}.

\ref{IIL:Psmooth:c2}:
Fix $x\in E$ and note that $k g(\fdot,x)\in \widehat C_\Delta(E^{k-1})$. The claim follows by~\ref{IIL:Psmooth:1}.

\ref{IIL:Psmooth:2}:
For  $p(\nu)=\langle g,\nu^k\rangle$ we have $|\partial_xp(\nu)|=|\langle kg(\fdot,x),\nu^{k-1}\rangle|\le k\|g\| |\nu|(E)^{k-1}<\infty$. Continuity of $x\mapsto \partial_xp(\nu)$ follows from the dominated convergence theorem and the fact that $E$ is Polish, and thus a sequential space.

\ref{IIL:Psmooth:3}: For monomials $p(\nu)=\langle g,\nu^k\rangle$ and $q(\nu)=\langle h,\nu^\ell\rangle$, we have $pq(\nu)=\langle g \otimes h,\nu^{k+\ell}\rangle$. Since for all $x\in E$ and $\nu\in M(E)$
$$(k+\ell)\langle g \otimes h(\fdot,x),\nu^{k+\ell-1}\rangle
=k\langle g(\fdot,x),\nu^{k-1}\rangle\langle h,\nu^\ell\rangle+\ell\langle g,\nu^k\rangle\langle h(\fdot,x),\nu^{\ell-1}\rangle,
$$
the claim follows by \ref{IIL:Psmooth:1}.

\ref{IIL:Psmooth:4}:
Observing that for $p(\nu):=\langle g,\nu^k\rangle$
$$p(\nu+\mu)=\sum_{\ell=0}^{k}\binom {k} \ell\int g(x_1,\ldots, x_{k}) \prod_{i=\ell+1}^{k} \nu(dx_i) \prod_{i=1}^\ell \mu(dx_i)$$
the result follows by \ref{IIL:Psmooth:1}.
\end{proof}

From Lemma~\ref{IIL:Psmooth}\ref{IIL:Psmooth:1} one can deduce the uniqueness of the representation~\eqref{IIeq:p(nu)}.

\begin{corollary} \label{IIC:pol uniqueness}
Suppose $p(\nu) = \sum_{k=0}^m \langle g_k,\nu^k\rangle$ equals zero for all $\nu\in M(E)$. Then $g_k=0$ for all $k$.
\end{corollary}

\begin{proof}
Let $x_1,\ldots,x_m\in E$ be arbitrary and differentiate $m$ times using Lemma~\ref{IIL:Psmooth}\ref{IIL:Psmooth:1} to get $m! g_m(x_1,\ldots,x_m) = \partial_{x_1x_2\cdots x_m}p(\nu) = 0$. Thus $g_m=0$. Now repeat this successively for $g_{m-1}$, $g_{m-2}$, $\ldots$, $g_0$.
\end{proof}

The following property turns out to be particularly useful in the context of the moment formula. In the finite-dimensional setting, the result states that every polynomial on the unit simplex has a homogeneous representative.

\begin{corollary}\label{IIrem5}
Every polynomial on $M(E)$ has a unique homogeneous representative on $M_1(E)$. That is, for every $p\in P$ with $\deg(p)=m$ there is a unique $g\in \widehat C_\Delta(E^{m})$ such that 
\[
p(\nu)=\langle g,\nu^m\rangle\qquad\text{for all}\ \nu\in M_1(E).
\]
\end{corollary}
\begin{proof}
Corollary \ref{IIC:pol uniqueness} yields a unique set of coefficients $g_0,\ldots,g_m$ with $g_k\in \widehat C_\Delta(E^k)$ and $p(\nu) = \sum_{k=0}^m \langle g_k,\nu^k\rangle$. The result now follows by setting $g:= \sum_{k=0}^m g_k\otimes 1^{\otimes (m-k)}$.
\end{proof}

\begin{remark}\label{IIrem3}
If we choose to work with coefficients in $\widehat C_0(E^k)$ instead of $\widehat C_\Delta(E^k)$ we would obtain the same class of polynomials on $M_1(E)$. This is because every $g\in \widehat C_\Delta(E^k)$ equals $\sum_{i=0}^k g_i\otimes 1^{\otimes (k-i)}$ for some $g_i\in\widehat C_0(E^i)$, and therefore $\langle g,\nu^k\rangle=\sum_{i=0}^k\langle g_i,\nu^i\rangle$ for all $\nu\in M_1(E)$. Indeed, the $g_i$ are given iteratively by
$$g_0:=g(\Delta,\ldots,\Delta)\quad\text{and}\quad g_i:=\binom k i \Big(g(\Delta,\ldots,\Delta,\fdot)-\sum_{j=0}^{i-1}g_j\otimes 1^{\otimes {(i-j)}}\Big).$$
 However, not every such polynomial admits a homogenous representative on $M_1(E)$ in the sense of Corollary~\ref{IIrem5}, unless $E$ is compact. An example is $1+\langle g,\nu\rangle$ with $g\in C_0(E)$ nonzero. The existence of homogeneous representatives leads to significant notational simplifications when $E$ is not compact (see Remark~\ref{IIrem4} for more details). This is the main reason for working with the spaces $\widehat C_\Delta(E^k)$.
\end{remark}

\subsection{Polynomials with regular coefficients}

The derivative map $x\mapsto\partial_x p(\nu)$ of a polynomial $p$ is only as regular as the coefficients of~$p$. This leads us to consider subspaces of polynomials with more regular coefficients.  Let
$
D \subseteq C_\Delta (E)
$
be a dense linear subspace containing the constant function 1 and define
\begin{equation} \label{IIeq:P^D}
P^D := {\rm span} \left\{1,\ \langle g,\nu\rangle^k\colon k\ge1, \ g\in D \right\}.
\end{equation}
Thus $P^D$ is the subalgebra of $P$ consisting of all (finite) linear combinations of the constant polynomial and ``rank-one'' monomials $\langle g\otimes\cdots\otimes g,\nu^k\rangle=\langle g,\nu\rangle^k$ with $g\in D$. Equivalently, $P^D$ consists of all polynomials $p(\nu)=\phi(\langle g_1,\nu\rangle,\ldots,\langle g_k,\nu\rangle)$ with $k\in\N$, $g_1,\ldots,g_k\in D$, and $\phi$ a polynomial on $\R^k$. 
\begin{lemma}\label{IIIlem1}
For any $p\in P^D$ and $\nu\in M(E)$, we have $\partial^k p(\nu)\in D^{\otimes k}$.  Moreover $P^D$ is dense in $C(M_1(E^\Delta))$. Here the elements of $P^D$ are viewed as functions on $M_1(E^\Delta)$ by first extending them to $M(E^\Delta)$ using Lemma \ref{IIL:Psmooth} \ref{IIL:Psmooth:c1} and then restricting them to $M_1(E^\Delta)$.\end{lemma}

\begin{proof}
For $p(\nu):=\phi(\langle g,\nu\rangle)$ where $\phi$ is polynomial we have $\partial^k p(\nu)=\phi^{(k)}(\langle g,\nu\rangle) g^{\otimes k}\in D^{\otimes k}$. Thus the first part of the result holds for all such~$p$, and by linearity for all $p\in P^D$. For the second part, continuity of polynomials follows by Lemma~\ref{IIL:Psmooth}\ref{IIL:Psmooth:c1}. Stone--Weierstrass and the fact that $D$ is densely contained in $C (E^\Delta)$ yield the density.
\end{proof}

\section{Optimality conditions}\label{IIstool}

We now develop optimality conditions for polynomials of measure arguments, which are instrumental when working with the positive maximum principle on $M_1(E^\Delta)$. Our first result, Theorem~\ref{IIL:KKT}, extends the classical first and second order Karush--Kuhn--Tucker conditions for functions on the finite-dimensional simplex (see e.g.~\cite{B:95}). It is derived by perturbing an optimizer $\nu_*\in M_1(E^\Delta)$ by shifting small amounts of mass to arbitrary points in $E^\Delta$. 
Our second result, Theorem~\ref{IIL:KKT2}, is obtained by deforming the optimizer $\nu_*$ using a group of isometries of $C_\Delta(E)$. The resulting condition is genuinely infinite-dimensional; see Lemma~\ref{IIL:triv_iso}. We will use the operator $\Psi$, which maps any function $g\colon E\times E\to\R^k$ to the function $\Psi(g)\colon E\times E\to\R^k$ given by
\begin{equation} \label{IIPhi(g)}
\Psi(g)(x,y) = \frac12 \left( g(x,x) + g(y,y)-2g(x,y) \right).
\end{equation}
Note that we use Lemma \ref{IIL:Psmooth}\ref{IIL:Psmooth:c1} to extend polynomials from $M_1(E)$ to $M_1(E^{\Delta})$.

\begin{theorem} \label{IIL:KKT}
Let $p\in P$ and $\nu_*\in M_1(E^\Delta)$ satisfy $p(\nu_*)=\max_{M_1(E^\Delta)} p$. Then the following first and second order optimality conditions hold:
\begin{enumerate}
\item\label{IIL:KKT:(i)} $\langle\partial p(\nu_*),\mu\rangle=\sup_{ E}\partial p(\nu_*),$
for all $\mu\in M_1(E^\Delta)$ such that $\supp(\mu)\subseteq\supp(\nu_*)$. 
In particular,
\begin{gather}
\partial_{x} p(\nu_*)=\sup_{E}\partial p(\nu_*) \quad \text{for all $x\in\supp(\nu_*)$}. \label{IIL:KKT:1}
\end{gather}
\item\label{IIL:KKT:(ii)} $\langle\partial^2p(\nu_*),\mu^2\rangle\leq 0$
for all signed measures $\mu\in M(E^\Delta)$ such that $\langle 1, \mu\rangle=0$ and $\supp(|\mu|)\subseteq\supp(\nu_*)$. In particular,
\begin{gather}
\Psi\big(\partial^2 p(\nu_*)\big)(x,y)\leq 0 \quad \text{for all $x,y\in \supp(\nu_*)$} \label{IIL:KKT:2}.
\end{gather}
\end{enumerate}
\end{theorem}

\begin{proof}
\ref{IIL:KKT:(i)}: Pick any $x\in\supp(\nu_*)$ and $y\in E^\Delta$. For each $n\in\N$, let $A_n$ be the ball of radius $1/n$ centered at $x$, intersected with $\supp(\nu_*)$. Then $\nu_*(A_n)>0$, and the probability measures $\mu_n:=\nu_*(\fdot\cap A_n)/\nu_*(A_n)$ converge weakly to $\delta_x$ as $n\to\infty$. Choose $\e_n\in(0,\nu_*(A_n))$. Then $\nu_* \geq \varepsilon_n\mu_n $ since for all $B \in \mathcal{B}(E)$
\[
\nu_*(B) -\e_n \frac{\nu_*(B\cap A_n)}{\nu_*(A_n)} \geq \frac{\nu_*(B \cap A_n)}{\nu_*(A_n)}( \nu_*(A_n)-\e_n) \geq 0.
\]
Hence $\nu_n:=\nu_* + \varepsilon_n(\delta_y - \mu_n)$ is a probability measure. Maximality of $\nu_*$ and Lemma~\ref{IIL:Psmooth}\ref{IIL:Psmooth:4} now give
\[
0 \geq p(\nu_n)-p(\nu_*) =\e_n\langle\partial p(\nu_*),\delta_y-\mu_n\rangle+o(\e_n).
\]
Dividing by $\varepsilon_n$, sending $n$ to infinity, and using that $x\mapsto\partial_xp(\nu_*)$ is bounded and continuous, we obtain $\partial_{x} p(\nu_*)\ge\partial_y p(\nu_*)$. We deduce \eqref{IIL:KKT:1}, which immediately implies~\ref{IIL:KKT:(i)}.

\ref{IIL:KKT:(ii)}: In addition to the above, suppose $y$ is in $\supp(\nu_*)$. Since we also have that $\supp(|\mu_n|)\subseteq\supp(\nu_*)$, we get $\langle\partial p(\nu_*),\delta_y-\mu_n\rangle=0$ due to \ref{IIL:KKT:(i)}. Maximality of $\nu_*$ and Lemma~\ref{IIL:Psmooth}\ref{IIL:Psmooth:4} then give
\[
0 \geq p(\nu_n)-p(\nu_*) = \frac12 \e_n^2\langle\partial^2 p(\nu_*),(\delta_y-\mu_n)^2\rangle+o(\e_n^2),
\]
and therefore $\langle\partial^2 p(\nu_*),(\delta_y-\delta_x)^2\rangle\le0$. More generally, consider measures of the form
\[
\nu_n:=\nu_*+\e_n\left( \sum_{i=1}^m\lambda_i\delta_{y_i} - \sum_{i=1}^m \gamma_i\mu_{i,n}\right)
\]
for some points $y_i\in\supp(\nu_*)$, convex weights $\lambda_1,\ldots,\lambda_m$ and $\gamma_1,\ldots,\gamma_m$, and $\mu_{i,n}$ constructed as $\mu_n$ above with $x$ replaced by $x_i\in\supp(\nu_*)$. Letting $\e_n$ decrease to zero sufficiently rapidly, the above argument gives $\langle\partial^2p(\nu_*),\mu^2\rangle\leq 0$ for the signed measure
\[
\mu = \sum_{i=1}^m\lambda_i\delta_{y_i} - \sum_{i=1}^m \gamma_i\delta_{x_i}.
\]
Passing to the weak closure yields \ref{IIL:KKT:(ii)} with the additional restriction that the positive and negative parts of $\mu$ are probability measures. The general case is obtained by scaling. Finally, since $\langle\partial^2 p(\nu_*),(\delta_y-\delta_x)^2\rangle=2\Psi(\partial^2 p(\nu_*))(x,y)$ we obtain \eqref{IIL:KKT:2}.
\end{proof}

\begin{remark}
Note the similarity between Theorem~\ref{IIL:KKT} and the classical Karush--Kuhn--Tucker conditions on the finite-dimensional simplex $\Delta^d$.
Let $f\in C^2(\R^d)$ and $x^*\in \Delta^d$ satisfy $f(x^*)=\max_{\Delta^d}f$. Then the first and second order (necessary) Karush--Kuhn--Tucker conditions on $\Delta^d$ hold:
\begin{enumerate}
\item For each $v\in \Delta^d$ such that $v_i=0$ whenever $x^*_i=0$,
$\nabla f(x^*)^\top v= \max_{j\in\{1,\ldots,d\}}\frac{\partial f}{\partial x_j}(x^*).$
\item For each $v\in\R^d$ such that $\mathbf1^\top v=0$ and $v_i=0$ whenever $x^*_i=0$,
$v^\top\nabla^2 f(x^*) v\leq0,$
where $\mathbf1:=(1,\ldots,1)^\top$.
\end{enumerate}
\end{remark}

\begin{remark}
Taking again $E=\{1, \ldots, d\}$ as example, the appearance of $\Psi$ in \eqref{IIL:KKT:2} can be understood as follows. Suppose $z\in\Delta^d$ maximizes a function $f\in C^2(\R^d)$ over $\Delta^d$. Then for every $i,j$ such that $z_i>0$ and $z_j>0$, we must have $(e_i-e_j)^\top \nabla^2 f(z)(e_i-e_j)\le0$, where $e_i$ is the $i$-th canonical unit vector. Indeed, otherwise $z\pm\varepsilon(e_i-e_j)$ would lie in $\Delta^d$ and give a higher function value for small $\varepsilon>0$. More explicitly, we must have
\[
\partial^2_{ii}f(z)+\partial^2_{jj}f(z)-2\partial^2_{ij}f(z)\le0,
\]
where the left hand side is equal to $2\Psi(\partial^2f(z))(i,j)$ on $E=\{1,\ldots,d\}$.
\end{remark}

For the remainder of this section, $D\subseteq C_\Delta(E)$ is a linear subspace, and $P^D$ is defined by \eqref{IIeq:P^D}.

Our next optimality condition is more subtle, in that it becomes trivial in the finite-dimensional case; see Lemma~\ref{IIL:triv_iso}. The basic observation is that a group of isometries $T_t$ of $C_\Delta(E)$ induces a flow of measures $\mu_t\in M_+(E^\Delta)$ via the formula $\langle g,\mu_t\rangle = \langle T_tg, \mu\rangle$ for every $g\in C_\Delta(E)$, where $\mu\in M_+(E^\Delta)$ is fixed. The value of a polynomial in its maximizer $\nu_*$ cannot be less than its value in $\nu_*-\mu+\mu_t$, for any $t$, and this leads to an optimality condition in terms of the group generator~$A$.

For example, if $E=\R$, the generator could be $Ag=\tau g'$ for some $\tau\in C_{\Delta}^1(\R)$. The isometries would then be $T_tg:=g(\phi(t,\fdot))$, where $\phi$ solves $\frac{d}{d t}\phi(t,x) = \tau(\phi(t,x))$ with initial condition $\phi(0,x)=x$. The corresponding flow of measures would consist of the pushforwards of $\mu$ with respect to $\phi(t,\fdot)$.  For more details see Lemma~\ref{lem1}.

The tensor notation $A\otimes A$ is used to denote the linear operator from $D\otimes D$ to $\widehat C_\Delta(E^2)$ determined by
$$(A\otimes A)(g\otimes g):=(Ag)\otimes (Ag)$$
 for a given linear operator $A\colon D\to  C_\Delta (E)$.

\begin{theorem}\label{IIL:KKT2}
Let $p\in P^D$ and $\nu_*\in M_1(E^\Delta)$ satisfy $p(\nu_*)=\max_{M_1(E^\Delta)} p$. Let $A$ be the generator of a strongly continuous group of positive isometries of $C_\Delta(E)$, and assume the domain of $A$ contains both $D$ and $A(D)$. Then
\[
\langle A^2(\partial p(\nu_*)), \mu\rangle + \langle (A\otimes A)(\partial^2 p(\nu_*)),\mu^2\rangle \le 0
\]
for every $\mu \in M_+(E^\Delta)$ with $\mu\le \nu_*$.
\end{theorem}

\begin{proof}
Let $\{T_t\}_{t\in\R}$ be the group generated by $A$. For any $\mu\in M_+(E^\Delta)$, the group induces a flow of measures $\mu_t\in M(E^\Delta)$ via the formula $\langle g,\mu_t\rangle = \langle T_tg, \mu\rangle$ for $g\in C_\Delta(E)$. The positivity and isometry property of $T_t$ implies that $\mu_t$ is nonnegative and has constant total mass $\mu_t(E^\Delta)=\mu(E^\Delta)$. Therefore, assuming henceforth that $\mu\le\nu_*$, it follows that $\nu_*+\mu_t-\mu$ is a probability measure.
Since $\|T_t g-g\|=O(t)$ for every $g\in D$, we have $\langle g,(\mu_t-\mu)^k\rangle=O(t^k)$ for every $g\in D^{\otimes k}$. Maximality of $\nu_*$ and Lemma~\ref{IIL:Psmooth}\ref{IIL:Psmooth:4} then give 
\begin{align}
0 &\ge p(\nu_*+\mu_t-\mu) - p(\nu_*) \nonumber \\
&= \langle\partial p(\nu_*),\mu_t-\mu\rangle + \frac12 \langle\partial^2 p(\nu_*),(\mu_t-\mu)^2\rangle+o(t^2) \nonumber \\
&=\langle(T_t-\id)\partial p(\nu_*),\mu\rangle + \frac12 \langle(T_t\otimes T_t-2\,T_t\otimes\id+\id\otimes\id)\partial^2p(\nu_*),\mu^2\rangle+o(t^2). \label{IIL:KKT2:eq1}
\end{align}
We claim that both $A$ and $-A$ satisfy the positive maximum principle on $E^\Delta$. Indeed, for $f\in D$ and $x\in E^\Delta$ with $f(x)=\max_{E^\Delta}f\ge0$, the positivity and isometry property give
\begin{equation} \label{IIeqL:KKT2:5}
T_t f(x) \le T_t f^+(x) \le \|T_t f^+\|=\|f^+\|=f(x).
\end{equation}
Thus $Af(x)=\lim_{t\downarrow0}(T_tf(x)-f(x))/t\le 0$ as well as $-Af(x)=\lim_{t\downarrow0}(T_{-t}f(x)-f(x))/t\le 0$, proving the claim.
Since $\partial_x p(\nu_*)=\sup_E\partial p(\nu_*)$ for all $x\in\supp(\nu^*)$ due to Theorem~\ref{IIL:KKT}, it follows that $A(\partial p(\nu_*))(x) = 0$ for all such $x$. As a result, using that $\supp(\mu)\subseteq\supp(\nu_*)$ and that the domain of $A$ contains $A(D)$, we get
\begin{equation} \label{IIL:KKT2:eq2}
\langle(T_t-\id)\partial p(\nu_*),\mu\rangle = \langle(T_t-\id-tA)\partial p(\nu_*),\mu\rangle = \frac12t^2\langle A^2( \partial p(\nu_*)),\mu\rangle+ o(t^2).
\end{equation}
Furthermore, using that
\[
(T_t\otimes T_t-2\,T_t\otimes\id+\id\otimes\id)(g\otimes g)=(T_tg-g)\otimes(T_tg-g)
\]
for all $g\in D$, we deduce that
\begin{equation} \label{IIL:KKT2:eq3}
\langle(T_t\otimes T_t-2\,T_t\otimes\id+\id\otimes\id)g,\mu^2\rangle = t^2 \langle (A\otimes A)g,\mu^2\rangle + o(t^2)
\end{equation}
for all $g\in D\otimes D$. Inserting \eqref{IIL:KKT2:eq2} and \eqref{IIL:KKT2:eq3} into \eqref{IIL:KKT2:eq1}, dividing by $t^2$, and sending $t$ to zero yields
\[
0 \ge \frac12 \langle A^2(\partial p(\nu_*)),\mu\rangle + \frac12 \langle (A\otimes A)\partial^2p(\nu_*),\mu^2\rangle.
\]
This completes the proof.
\end{proof}

 \begin{remark}\label{IIrem7}
We claim that for $A$ as in Theorem~\ref{IIL:KKT2}, the operator $A^2$ satisfies the positive maximum principle on $E^\Delta$. Indeed, let $f\in D$ and $x\in E^\Delta$ with $f(x)=\max_{E^\Delta}f\ge0$. Then, as in \eqref{IIeqL:KKT2:5} and with the same notation, we have $T_tf(x)\le f(x)$, and $Af(x)=0$ since both $A$ and $-A$ satisfy the positive maximum principle on $E^\Delta$. Hence $A^2f(x) = \lim_{t\downarrow0}(T_tf(x)-f(x)-Af(x))/t\le0$, which proves the claim.
\end{remark}

The following lemma illustrates the pure infinite-dimensional nature of the condition provided in Theorem~\ref{IIL:KKT2}.

\begin{lemma} \label{IIL:triv_iso}
Let $A$ be the generator of a strongly continuous group of positive isometries of $C_\Delta(E)$. If the domain of $A$ is all of $C_\Delta(E)$, then $A=0$. This is in particular the case if $A$ is bounded or $E$ consists of finitely many points.
\end{lemma}

\begin{proof}
Both $A$ and $-A$ satisfy the positive maximum principle on $E$, and $A1=0$. Therefore Lemma~\ref{IIML_L1} implies that $A$ and $-A$ are both of the form \eqref{IIeqML_L1} with $B=\pm A$. As a result,
\[
0 = Ag(x) - Ag(x) = \int ( g(\xi)-g(x))(\nu_A + \nu_{-A})(x,d\xi)
\]
for all $x\in E$ and $g\in C(E^\Delta)$. This implies that $1_{\{x\}^c}(\xi)\nu_A(x,d\xi)$ and $1_{\{x\}^c}(\xi)\nu_{-A}(x,d\xi)$ are zero for all $x\in E$ and hence that $A=0$. Since each linear operator on a finite-dimensional vector space is bounded, and the domain of a bounded operator on $C_\Delta(E)$ can be extended to all of $C_\Delta(E)$, the second part follows.
\end{proof}

\section{Polynomial operators} \label{sec_pol_op}

Let $E$ be a locally compact Polish space. We now define polynomial operators, which constitute a class of possibly unbounded linear operators acting on polynomials. They are not defined on all of $P$ in general, but only on the subspace $P^D$ for some dense subspace $D \subseteq C_\Delta (E)$; see \eqref{IIeq:P^D}. An analog of this notion has appeared previously in connection with finite-dimensional polynomial processes; see e.g.~\cite{CKT:12,FL:16,CLS:17}.

\begin{definition}\label{IIdef1}
Fix $\s\subseteq M(E)$. A linear operator $L\colon P^D\to P$ is called $\s$-polynomial if for every $p\in P^D$ there is some $q\in P$ such that $q|_{\s}=Lp|_{\s}$ and
$$\deg(q) \le \deg(p).$$
\end{definition}

Given a linear operator $L\colon P^D\to P$, its associated carr\'e-du-champ operator is the symmetric bilinear map $\Gamma\colon P^D\times P^D\to P$ defined by
\begin{equation} \label{IIeq:Gamma}
\Gamma(p,q)=L(pq)-pLq-qLp.
\end{equation}
The carr\'e-du-champ operator gives information about the quadratic variation of the martingales appearing in the martingale problem for the operator $L$. It also gives information about path continuity of solutions to such martingale problems. We return to this issue in Lemma~\ref{IIlem5}, which roughly speaking states that path continuity holds precisely when the carr\'e-du-champ operator $\Gamma$ is a derivation, which is defined as follows.

\begin{definition}
Fix $\s\subseteq M(E)$. A symmetric bilinear map $\Gamma\colon P^D\times P^D\to P$ is called an $\s$-derivation if for all $p,q,r\in P^D$, $\Gamma(pq,r)=p\Gamma(q,r)+q\Gamma(p,r)$ on $\s$.
\end{definition}

For a finite-dimensional diffusion it is known that its generator is polynomial if and only if the drift and diffusion coefficients are polynomial of first and second degree, respectively; see~\cite{CKT:12} and \cite{FL:16}. The following result is the generalization of this fact to the probability-valued setting. The proof is given in Section~\ref{IIE}.

\begin{theorem}\label{IIT:Lpol}
Let $L\colon P^D\to P$ be a linear operator. Then $L$ is $M_1(E)$-polynomial and its carr\'e-du-champ operator $\Gamma$ is an $M_1(E)$-derivation if and only if 
\begin{equation}\label{IIeq:T:Lpol}
\begin{gathered}
Lp(\nu) = \big\langle B(\partial p(\nu)), \nu\big\rangle + \frac{1}{2} \big\langle Q(\partial^2p(\nu)), \nu^2\big\rangle, \quad \nu\in M_1(E),\\
\begin{minipage}[c][2em][c]{.8\textwidth}
\begin{center}
for some linear operators $B\colon D\to  C_\Delta (E)$ and $Q\colon D\otimes D\to \widehat C_\Delta (E^2)$.
\end{center}
\end{minipage}
\end{gathered}
\end{equation}
In this case, $B$ and $Q$ are uniquely determined by $L$.
\end{theorem}

An analogue of Theorem \ref{IIT:Lpol} holds for $L$ being $\s$-polynomial, where $\s$ is an arbitrary subset of $M(E)$; see Theorem~\ref{IIT:Lpol gen}.

\begin{example}[The Fleming-Viot generator]\label{IIex2}
Let $E=\R$ and $D=C_\Delta^2(\R)$. The Fleming--Viot diffusion was introduced by \cite{FV:79} and subsequently studied by several other authors. This process takes values in $M_1(\R)$, and its generator $L$ acts on polynomials $p\in P^D$ by
\[
Lp(\nu) = \int_E B(\partial p(\nu) )(x) \nu(dx) + \frac{1}{2} \int_{E^2} \partial_{xy}^2 p(\nu) \nu(dx) (\delta_x(dy) - \nu(dy) ), \quad \nu\in M_1(E),
\]
where $B g:=\frac{1}{2} \sigma^2g''$ for some $\sigma\in \R$.
This is an $M_1(\R)$-polynomial operator of the form \eqref{IIeq:T:Lpol}, where $Q=\Psi$ as defined in~\eqref{IIPhi(g)}.
For more details, see Chapter 10.4 of \cite{EK:05}.
\end{example}

Corollary~\ref{IIrem5} states that any polynomial on $M_1(E)$ has a unique homogeneous representative. 
Therefore, an operator $L$ satisfying \eqref{IIeq:T:Lpol} actually maps any monomial $\langle g,\nu^k\rangle$ to a unique monomial $\langle h,\nu^k\rangle$ on $M_1(E)$. This induces an operator $L_k$ acting on the corresponding coefficients by $L_kg:= h$. The operators $ L_1,L_2,\ldots$ are the key objects needed to compute conditional moments of polynomial diffusions corresponding to $L$.

\begin{definition} \label{IIeq:L on Dk}
Let $L\colon P^D\to P$ satisfy \eqref{IIeq:T:Lpol}. The \emph{$k$-th dual operator} of $L$ is defined as the unique linear operator $L_k\colon D^{\otimes k}\to\widehat C_\Delta (E^k)$ determined by
\begin{equation}\label{eqn15}
Lp(\nu)=\langle L_kg, \nu^k\rangle,\qquad\nu\in M_1(E),
\end{equation}
for every $p(\nu)=\langle g,\nu^k\rangle$ with $g\in D^{\otimes k}$.
\end{definition}

Because of \eqref{IIeq:T:Lpol}, the $k$-th dual operator $L_k$ can be written
\begin{equation}\label{eqn4}
L_k=kB\otimes\id^{\otimes (k-1)}+\frac {k(k-1)} 2 Q\otimes \id^{\otimes (k-2)},
\end{equation}
where the tensor notation $B_1\otimes\ldots\otimes B_N$ is used to denote the linear operator from $D^{\otimes k}$ to $\widehat C_\Delta(E^{k})$ determined by $(B_1\otimes\ldots\otimes B_N)(g^{\otimes k}):=B_1(g^{\otimes n_1})\otimes\ldots\otimes B_N(g^{\otimes n_N})$ for given linear operators $B_i\colon D^{\otimes n_i}\to \widehat C_\Delta (E^{n_i})$ with $n_1+\cdots+n_N=k$. 
More explicitly, we have
\[
L_k=B_k+ Q_k
\]
where $B_k$ and $Q_k$ are defined by
\begin{equation}\label{IIeqn27}
B_k g:=\sum_{i=1}^kB^{(i)} g\qquad\text{and}\qquad Q_kg:=\frac 1 2 \sum_{i,j=1}^kQ^{(ij)}g
\end{equation}
for $B^{(i)}g(x) := B g(\ldots,x_{i-1},\fdot,x_{i+1},\ldots)(x_i)$ and 
$$Q^{(ij)}g(x) := Q\big(g(\ldots,x_{i-1},\fdot,x_{i+1},\ldots,x_{j-1},\fdot,x_{j+1},\ldots)\big)(x_i,x_j).
$$

\begin{remark}\label{IIrem4}
Observe that without the existence of an homogeneous representative (guaranteed by Corollary~\ref{IIrem5}), expression \eqref{eqn15} would read
$$Lp(\nu)=\langle L^k_kg, \nu^k\rangle+\langle L_k^{k-1}g, \nu^{k-1}\rangle+\cdots+L_k^0g,\qquad\nu\in M_1(E),$$
and the $k$-th  dual operator would thus consist in a $(k+1)$-tuple of operators $L_k^k,\ldots, L_k^0$. In the context of the moment formula, as stated in Theorem \ref{IIthm:new} below, the PIDE of \eqref{IIeq:PIDE} would then translate to a system of $(k+1)$ PIDEs. If one is interested in studying jump-diffusions taking value in other subspaces of $M(E)$, as e.g.~$M_+(E)$,  a homogeneous representative can no longer be found and one has to deal with systems of PIDEs to compute the moments.
\end{remark}

\section{Existence and uniqueness of polynomial diffusions on $M_1(E)$} \label{sec_ex_un}

Let $E$ be a locally compact Polish space, $D$ a dense linear subspace of $C_\Delta (E)$ containing the constant function $1$, and $L\colon P^D\to P$ a linear operator. In this section we study existence and uniqueness of $M_1(E)$-valued polynomial diffusions, and derive the moment formula.

An $M_1(E)$-valued process $X$ with c\`adl\`ag paths defined on some filtered probability space {$(\Omega, \Fcal,(\Fcal_t)_{t\geq0}, \P)$} is called a {\em solution to the martingale problem for $L$} with initial condition $\nu\in M_1(E)$ if $X_0=\nu$ $\P$-a.s.~and
\begin{equation}\label{IIeqnN}
N^p_t = p(X_t) - p(X_0) - \int_0^t Lp(X_s) ds
\end{equation}
defines a martingale for every $p\in P^D$. Uniqueness of solutions to the martingale problem is always understood in the sense of law. The martingale problem for $L$ is {\em well--posed} if for every $\nu\in M_1(E)$ there exists a unique $M_1(E)$-valued solution to the martingale problem for $L$ with initial condition~$\nu$. We are interested in solutions with continuous paths (with respect to the topology of weak convergence) corresponding to polynomial operators.

\begin{definition}
Let $L$ be $M_1(E)$-polynomial. Any continuous solution to the martingale problem for $L$ is called a
\emph{probability-valued polynomial diffusion}.
\end{definition}

The following lemma relates path continuity of solutions to the martingale problem to the carr\'e-du-champ operator being a derivation. This explains why we consider derivations in Theorem~\ref{IIT:Lpol}.

\begin{lemma}\label{IIlem5}
If the carr\'e-du-champ operator $\Gamma$ of $L$ is an $M_1(E)$-derivation, then any solution to the martingale problem for $L$ has continuous paths. Conversely, if for every initial condition $\nu\in M_1(E)$ there is a solution to the martingale problem for $L$ with continuous paths, then the carr\'e-du-champ operator $\Gamma$ associated to $L$ is an $M_1(E)$-derivation.
\end{lemma}

\begin{proof}
Let $X$ be a solution to the martingale problem for $L$. By Proposition~2 in \cite{BE:85}, the real-valued process $p(X)$ is continuous for every $p\in P^D$, in particular for every linear monomial $p(\nu)=\langle h,\nu\rangle$ with $h\in D$. Since $D$ is dense in $C_\Delta (E)$, we can conclude that $X$ is continuous with respect to the topology of weak convergence on $M_1(E)$.

Conversely, if $X$ is a solution to the martingale problem for $L$ with continuous paths, then, by Lemma \ref{IIL:Psmooth}\ref{IIL:Psmooth:c1}, the map $t\mapsto p(X_t)$ is continuous for all $p\in P^D$. The result now follows by Proposition 1 in \cite{BE:85}.
\end{proof}

\subsection{Moment formula and uniqueness in law}\label{sec_moments_uniqueness}

Polynomial diffusions are of interest in applications because they generally satisfy a {\em moment formula}, which allows moments of the process to be computed tractably. If $E$ is a finite set, the moment formula always holds, but technical conditions, in particular on the dual operators, are needed in the general case. For details regarding operators and semigroups, we refer e.g.~to \cite{EK:05}.

\begin{theorem}\label{IIthm:new}
Suppose $L$ satisfies \eqref{IIeq:T:Lpol} and fix $k\in\N$. Assume that the $k$-th dual operator $L_k$ is closable, and let $g$ be in the domain of its closure $\overline L_k$.
 Suppose that there is a solution $u\colon\R_+\times E^k\to\R$ of 
\begin{equation}\label{IIeq:PIDE}
\begin{aligned}
\frac{\partial u}{\partial t}(t,x) &= \overline L_k u(t,\fdot)(x),	&&\qquad (t,x)\in\R_+\times E^k,\\
u(0,x)&= g(x),									&&\qquad x\in E^k,
\end{aligned}
\end{equation}
and suppose that $\sup_{t\in[0,T]} \|\overline L_k u(t,\fdot)\|<\infty$ for all $T\in \R_+$. In particular, $u(t,\fdot)$ is assumed to be in the domain of $\overline L_k$ for all $t\geq0$. Then for any continuous solution $X$ to the martingale problem for $L$, one has the moment formula
\begin{equation}\label{eqn3}
\E\big[\langle g, X^k_T\rangle \mid \Fcal_t \big]=\langle u(T-t, \cdot), X^k_t\rangle.
\end{equation}
\end{theorem}

\begin{proof}
We will follow the proof of Theorem 4.4.11 in \cite{EK:05} and extend it to obtain also the formula for the \emph{conditional} moments. Fix $T\in \R_+$, $t\in [0,T]$, and $A\in \Fcal_t$.
Define for all $(s_1,s_2)\in[0,T-t]\times [0,T-t]$ define 
$f(s_1, s_2) := \E[\langle u(s_1,\fdot), X_{t+s_2}^k \rangle  1_A]$.
Fix $s_2\in[0,T-t]$. Equation \eqref{IIeq:PIDE} and the fundamental theorem of calculus then yield 
$$
f(s_1, s_2) -f(0, s_2)
=\E[
\langle   u(s_1,\fdot)-u(0,\fdot), X_{t+s_2}^k \rangle  1_A]=\int_0^{s_1} \E[
\langle   \overline L_ku(s,\fdot), X_{t+s_2}^k \rangle 1_A]ds.
$$
Fix then $s_1\in[0,T-t]$. Since $u(t,\fdot)$ is in the domain of $\overline L_k$ for all $t\in \R_+$, \eqref{IIeqnN} yields
\begin{align*}
f(s_1, s_2) -f(s_1, 0)
&=\E[\E[
\langle   u(s_1,\fdot), X_{t+s_2}^k \rangle -\langle   u(s_1,\fdot), X_{t}^k \rangle 
|\Fcal_t] 1_A]\\
&=\int_0^{s_2} \E[
\langle   \overline L_ku(s_1,\fdot), X_{t+s}^k \rangle 1_A]ds.
\end{align*}
Since $\sup_{s_1,s_2\in[0,T-t]} \big|\E[
\langle   \overline L_ku(s_1,\fdot), X_{t+s_2}^k \rangle 1_A]\big|
\leq \sup_{s_1\in[0,T]}\| \overline L_ku(s_1,\fdot)\|<\infty$, we can then conclude that both $f(\fdot, s_2)$ and  $f(s_1, \fdot)$ are absolutely continuous with bounded derivatives. Lemma 4.4.10 in \cite{EK:05} then yields $f(T-t,0)-f(0,T-t)=0$, and the result follows.
\end{proof}

In order to avoid confusion, for the rest of the section we denote by $u_g$ the solution of \eqref{IIeq:PIDE} with initial condition $u_g(0,\fdot)=g$.

In most of the cases of interest (see Remark~\ref{IIrem:maxprinciple}\ref{IIrem10} below) the operator $L_k$ satisfies the positive maximum principle on $E^k$, for each $k\in \N$. If this is the case, the existence of a solution $u_g$ of \eqref{IIeq:PIDE} satisfying the conditions of Theorem~\ref{IIthm:new} for sufficiently many $g$, is essentially equivalent to the fact that $\overline L_k$ generates a strongly continuous positve contraction semigroup on $\widehat C_\Delta(E^k)$ or in other words that it is the generator of a Feller process on $E^k$. We state this precisely in the following remark.

\begin{remark}\label{rem2}
Let $L$ satisfy \eqref{IIeq:T:Lpol} and  let $X$ denote a solution to the corresponding martingale problem with initial condition $X_0=\nu\in M_1(E)$.
Assume that the corresponding $k$-th dual operator $L_k$ satisfies the positive maximum principle on $(E^\Delta)^k$ (which in particular implies that $L_k$ is closable), for each $k\in \N$. 

Let $D_0$ be a dense subset of the domain of $\overline L_k$ and suppose that the conditions of Theorem~\ref{IIthm:new} hold true for all $g\in D_0$. 
By Proposition~1.3.4 of \cite{EK:05}, if we additionally have that
$t\mapsto \overline L_ku_g(t,\fdot)$ is continuous,
then $\overline L_k$ is the generator of a strongly continuous contraction semigroup $\{Y^k_t\}_{t\geq0}$ on $\widehat C_\Delta(E^k)$ and $Y_t^kg=u_g(t, \fdot)$. In this case the moment formula reads as
\[
\E\big[\langle g, X^k_T\rangle \mid \Fcal_t \big]=\langle Y_{T-t}^kg, X^k_t\rangle, \quad \text{for all } g \in \widehat C_\Delta(E^k).
\]
Conversely, if $\overline L_k$ is the generator of a strongly continuous contraction semigroup $\{Y^k_t\}_{t\geq0}$ on $\widehat C_\Delta(E^k)$, then for all $g$ in the domain of $\overline L_k$ the map $u_g(t,x):=Y^k_t g(x)$ satisfies the conditions of Theorem~\ref{IIthm:new}. By the Hille--Yosida theorem, this is for instance the case if the range of $\lambda-L_k$ is dense in $\widehat C_\Delta(E^k)$ for some $\lambda>0$.
In this case, Corollary 4.2.8 in \cite{EK:05} yields a solution $Z^{(k)}$ (without loss of generality defined on the same probability space as $X$) to the martingale problem for $L_k$ with values in $(E^\Delta)^k$ and satisfying  $Y_t^kg(x)=\E[g(Z_t^{(k)})|Z_0^{(k)}=x]$. The moment formula then yields
\begin{equation}\label{IIeq:momentformulaZ}
\E[g(Z_t^{(k)})|Z_0^{(k)}\sim \nu^k]= \E[ \langle g,X_t^k \rangle].
\end{equation}
This gives an alternative interpretation to \eqref{eqn3}, namely that the PIDE in \eqref{IIeq:PIDE} is the Feynman-Kac PIDE associated to the $k$-dimensional process Markov process $Z^{(k)}$.  Note that in the case of a finite state space $E$, \eqref{IIeq:PIDE} reduces to an ODE and $\overline L_k$ is automatically the generator of $k$-dimensional Markov chain.
\end{remark}

As in the finite-dimensional case, the moment formula yields well--posedness of the martingale problem.

\begin{corollary} \label{IIcor:uniqueness}
Suppose $L$ satisfies \eqref{IIeq:T:Lpol}, and let $X$ be a continuous solution to the martingale problem for $L$ with initial condition $\nu\in M_1(E)$. If the moment formula \eqref{eqn3} holds for all $g\in D^{\otimes k}$ and $k\in\N$, then the law of $X$ is uniquely determined by $L$ and $\nu$.
\end{corollary}

\begin{proof}
By the moment formula \eqref{eqn3} we have $\E[\langle g, X^k_T\rangle]=\langle u_g(T,\fdot), \nu^k\rangle$ for all $k\in\N$ and  $g\in D^{\otimes k}$.
Since $g\mapsto u_g$ is determined by  $L$, Lemma~\ref{IIIlem1} yields that  the one-dimensional distributions of $X$ are uniquely determined by $L$ and $\nu$. The conclusion follows by Theorem~4.4.2 in \cite{EK:05}.
\end{proof}

\subsection{Existence and well-posedness}\label{IIs40}

Our first main result of this section gives abstract sufficient conditions for existence of solutions to the martingale problem. Applications of this result are discussed in Section \ref{IIsec:examples}. Recall that $E$ is throughout a locally compact Polish space.

\begin{theorem}\label{IImainthm5}
Let $D\subseteq C_\Delta(E)$ be a dense linear subspace containing the constant function~$1$. Let $L\colon P^D\to P$ be a linear operator satisfying~\eqref{IIeq:T:Lpol}, where
\begin{enumerate}
\item\label{it7} $B$ is $E$-conservative and satisfies $B1=0$,
\item $Q$ is given by
\[
Q(g) = \cC  \Psi(g) + \sum_{i=1}^n (A_i\otimes A_i)(g), \qquad g\in D\otimes D,
\]
where $\cC \colon E^2\to\R$ is a nonnegative symmetric function and, for $i=1,\ldots,n$, $A_i$ is the generator of a strongly continuous group of positive isometries of $C_\Delta(E)$, and the domain of $A_i$ contains both $D$ and $A_i(D)$,
\item\label{itiii} $B - \frac12\sum_{i=1}^n A_i^2$ satisfies the positive maximum principle on $E^\Delta$.
\end{enumerate}
Then $L$ is $M_1(E)$-polynomial and its martingale problem has a solution with continuous paths for every initial condition $\nu\in M_1(E)$.
If in addition the moment formula \eqref{eqn3} holds for all $g\in D^{\otimes k}$ and $k\in\N$, then the martingale problem for $L$ is well--posed.
\end{theorem}

    Note that \eqref{IIeq:T:Lpol} imposes the implicit condition on $\cC $ that $\cC \Psi(g)$ must lie in $\widehat C_\Delta(E^2)$ for every $g\in D\otimes D$. If $D=C_\Delta(E)$, then $\cC $ is necessarily bounded, as is seen from Theorem~\ref{IImainthm3} below. However, this does not hold for general $D\subseteq C_\Delta(E)$, as one can see by considering $E=\R$, $D\subseteq C^1_\Delta(\R)$, and $\cC(x,y)=|x-y|^{-1}1_{\{x\neq y\}}$.

\begin{proof}
Theorem~\ref{IIT:Lpol} shows that $L$ is $M_1(E)$-polynomial. Lemma~\ref{IIIlem8} yields existence of a solution to the martingale problem for any initial condition (necessarily with continuous paths due to Lemma~\ref{IIlem5}) once we check that $L$ satisfies the positive maximum principle on $M_1(E^\Delta)$. Let therefore $\nu_*\in M_1(E^\Delta)$ be a maximizer of $p\in P^D$ over $M_1(E^\Delta)$. The optimality conditions in Theorem~\ref{IIL:KKT} yield
\[
\partial_xp(\nu_*)=\sup_E\partial p(\nu_*)\quad\text{and}\quad \Psi\big(\partial^2 p(\nu_*)\big)({x,y})\leq 0, \quad x,y\in\supp(\nu_*).
\]
Therefore, since $B - \frac12\sum_{i=1}^n A_i^2$ satisfies the positive maximum principle and $\cC $ is nonnegative, we get
\[
Lp(\nu_*) \le \frac12 \sum_{i=1}^n \left( \langle A_i^2(\partial p(\nu_*)), \nu_*\rangle + \langle (A_i\otimes A_i)(\partial^2p(\nu_*)),\nu_*^2\rangle\right).
\]
The optimality condition in Theorem~\ref{IIL:KKT2} now yields $Lp(\nu_*)\le0$. This proves the positive maximum principle and thus the existence statement. The assertions regarding the moment formula and well--posedness follow from Theorem~\ref{IIthm:new} and Corollary~\ref{IIcor:uniqueness}.
\end{proof}

\begin{remark} \label{IIrem:maxprinciple}
\begin{enumerate}
\item \label{it2}

With regard to item \ref{itiii} in Theorem \ref{IImainthm5}, note that a linear operator $\Gcal: D\to C_\Delta(E)$ satisfies the positive maximum principle on $E^\Delta$ if and only if $\Gcal$ satisfies the positive maximum principle on $E$ and $\Gcal g (\Delta)\geq 0$ for every nonnegative $g\in C_0(E)\cap D$. In many cases of interest, for instance  $E\subseteq\R^d$ and $D\subseteq \R+C_c(E)$, the positive maximum principle on $E$ implies the positive maximum principle on $E^\Delta$.

\item\label{it6} Let us also remark, that the $k$-th dual operator $\mathcal{G}_k$ associated to $\langle \mathcal{G} (\partial p(\nu)), \nu \rangle$ satisfies the positive maximum principle on $(E^\Delta)^k$ if it holds for $\mathcal{G}$ on $E^{\Delta}$. Indeed, if $x^{\ast} \in (E^\Delta)^k$ is a maximum of $g$, then $x_i^{\ast}$ is a maximum of $g(\ldots, x_{i-1}^*, \fdot,x_{i+1}^*,\ldots)$. Hence $\mathcal{G}_k$ given by
\[
\mathcal{G}_kg=k \Gcal\otimes \id^{\otimes  (k-1)}g=\sum_{j=1}^k \mathcal{G}^{(j)}g,
\]
where we use the same notation as in \eqref{IIeqn27}, clearly satisfies the positive maximum principle on $(E^\Delta)^k$.

\item \label{IIrem10}
Consider the setting and the assumptions of Theorem~\ref{IImainthm5} and define
\begin{align*}
G_k&:=k\Big(B- \frac12\sum_{i=1}^n A_i^2\Big)\otimes \id^{\otimes (k-1)},\qquad
C_k:=\frac {k(k-1)}2 (\cC  \Psi)\otimes \id^{\otimes(k-2)},\\ 
T_k&:=k\Big(\frac12\sum_{i=1}^n A_i^2\Big)\otimes\id^{\otimes(k-1)}+\frac{k(k-1)}2\Big(\sum_{i=1}^n (A_i\otimes A_i)\Big)\otimes \id^{\otimes(k-2)}.
\end{align*}
Note that by \eqref{eqn4} we have $L_k=G_k+C_k+T_k$. We claim that $G_k$, $C_k$, $T_k$, and hence $L_k$, satisfy the positive maximum principle on $(E^\Delta)^k$.

By item \ref{itiii}  in Theorem \ref{IImainthm5}, $B- \frac12\sum_{i=1}^n A_i^2$ satisfies the positive maximum principle on $E^\Delta$, whence by \ref{it6} it holds also for $G_k$ on  $(E^\Delta)^k$. The form of $\Psi$ and the nonnegativity of $\cC $ guarantee that this is also the case for $C_k$.  Finally, since $T_k=\sum_{i=1}^n \frac 12 (\sum_{j=1}^n A^{(j)}_i)^2$ where
$A_i^{(j)}g(x) = A_ig(\ldots , x_{jâ-1},\fdot, x_{j+1}, \ldots)(x_j),$
 Remark~\ref{IIrem7} yields the positive maximum principle on $(E^\Delta)^k$ also for $T_k$ and thus all together for $L_k$. 
\end{enumerate}
\end{remark}

The following result gives a useful condition for uniqueness when all the operators $A_i$ are zero. Due to Lemma~\ref{IIL:triv_iso} this happens, for instance, if $D=C_\Delta(E)$ and in particular if $E$ consists of finitely many points. An example where uniqueness holds when those operators are not all zero is given in Example~\ref{ex1}.

\begin{lemma}\label{IIlem3}
Consider setting and assumptions of Theorem~\ref{IImainthm5}, and assume that $A_i=0$ for all $i$. Assume additionally that $\cC $ is bounded and  $B$ is closable and its closure is the generator of a strongly continuous contraction semigroup on $C_\Delta(E)$. Then the moment formula \eqref{eqn3} holds for all $g\in \widehat C_{\Delta}(E^k)$ and $k\in\N$.
\end{lemma}
Since $B$ satisfies the positive maximum principle on $E^\Delta$ by Theorem~\ref{IImainthm5}\ref{itiii}, the Hille--Yosida theorem guarantees that the conditions of the lemma are satisfied whenever $\lambda-B$ has dense range in $C_\Delta(E)$ for some $\lambda>0$.

\begin{proof}
Let $\{Y^1_t\}_{t\ge0}$ be the semigroup corresponding to $\overline B$. Fix any $k\in\N$ and let $B_k$ and $Q_k$ be as in \eqref{IIeqn27}. It is straightforward to check that $B_k$ is the restriction to $D^{\otimes k}$ of the generator of the strongly continuous contraction semigroup $\{\overline{(Y^1_t)^{\otimes k}}\}_{t\ge0}$ on $\widehat C_\Delta(E^k)$. Moreover, one has the estimate
\[
\|Q_kg\| \le {k(k-1)}\|\cC \| {\|g\|}, \quad g\in\widehat C_\Delta(E^k),
\]
whence $Q_k$ is a bounded operator. It follows as in Theorem~1.7.1 and Corollary~1.7.2 in \cite{EK:05} that $L_k=B_k+Q_k$ is closable and its closure is the generator of a strongly continuous contraction semigroup on $\widehat C_\Delta(E^k)$. By Remark~\ref{rem2} and Theorem~\ref{IIthm:new} the result follows.
\end{proof}

While Theorem~\ref{IImainthm5} only gives sufficient conditions for existence, the result is sharp. Indeed, we now show that if $D=C_\Delta(E)$, no other polynomial specifications exist. For instance, this is the case if $E$ is a finite set. The following theorem, which is our second main result of this section, makes this precise. The proof is given in Section~\ref{appB}.

\begin{theorem}\label{IImainthm3}
Let $D=C_\Delta(E)$ and let $L\colon P^D\to P$ be a linear operator. Then $L$ is $M_1(E)$-polynomial, its martingale problem is well posed, and all solutions have continuous paths, if and only if $L$ satisfies~\eqref{IIeq:T:Lpol} with
\begin{equation} \label{IIeqn1_2}
Bg = \int \left( g(\xi)-g(\fdot)\right)\nu_B(\fdot,d\xi)\quad\text{and}\quad Qg = \cC  \Psi (g),
\end{equation}
where $\nu_B$ is a nonnegative, finite kernel from $E$ to $E$, and $\cC \colon(E^\Delta)^2\to\R$ is nonnegative, symmetric, bounded, and continuous on $(E^\Delta)^2\setminus\{x=y\}$. In this case, for each $k\in\N$ the $k$-th dual operator $L_k$ satisfies the hypothesis of Theorem~\ref{IIthm:new}, and the moment formula \eqref{eqn3} holds for all $g\in \widehat C_{\Delta}(E^k)$. Moreover, $B$ and $Q$, and hence each $L_k$, are bounded operators.
\end{theorem}

As in Theorem~\ref{IImainthm5}, condition \eqref{IIeq:T:Lpol} imposes implicit conditions on the different parameters. This is the case for the measure $\nu_B$, which in particular needs to satisfy $\int g(\xi)-g(\fdot) \nu_B(\fdot,d\xi)\in C_\Delta(E)$ for all $g\in C_\Delta(E)$. This is condition is clearly satisfied if the map from $E$ to $M_+(E)$ given by $x\mapsto \nu_B(x,\fdot)$ is continuous. However the converse fails to be true as one can see by considering the following kernel
$$\nu_B(x,d\xi)=\delta_{\phi(x)} 1_{\{\phi(x)\neq x\}},$$
for some continuous $\phi:E\to E$ such that $\phi\neq \id$.

\begin{corollary}\label{cor2}
Let $D\subseteq C_\Delta(E)$ be a dense linear subspace containing the constant function 1 and let $L$ satisfy~\eqref{IIeq:T:Lpol} with $B$ and $Q$ as in Theorem~\ref{IImainthm3}. Then $L$ is $M_1(E)$-polynomial, its martingale problem is well--posed, and all solutions have continuous paths. Moreover, the moment formula \eqref{eqn3} holds for all $g\in D^{\otimes k}$ and $k\in\N$.
\end{corollary}
\begin{proof}
Since by Theorem~\ref{IImainthm3}  each $L_k$ is bounded, the operator $L$ can be uniquely extended to $P^{C_\Delta(E)}$. The result then follows by the same theorem.
\end{proof}

The last main result of this section characterizes  probability-valued polynomial martingales. An $M_1(E)$-valued process $X$ is called a martingale if $\langle g, X\rangle$ is a martingale for every $g\in C_\Delta(E)$. Note that, unlike Theorem~\ref{IImainthm5}, the conditions are both necessary and sufficient, regardless of the choice of domain~$D$.

\begin{theorem}\label{IIthm1}
Let $D\subseteq C_\Delta(E)$ be a dense linear subspace containing the constant function~$1$. Let $L\colon P^D\to P$ be a linear operator. Then $L$ is $M_1(E)$-polynomial, its martingale problem has a solution for any initial condition, and every solution is a martingale with continuous paths, if and only if $L$ satisfies \eqref{IIeq:T:Lpol} with
\[
B=0\qquad\text{and}\qquad Q= \cC \Psi
\]
for some nonnegative symmetric function $\cC \colon E^2\to\R$. In this case, if in addition  $\cC $ is bounded, the martingale problem is well--posed.
\end{theorem}

\begin{proof}
To prove the forward implication, first note that Lemma \ref{IIlem5} and Theorem \ref{IIT:Lpol} imply that $L$ satisfies \eqref{IIeq:T:Lpol}. To see that $B=0$, pick any $g\in D$ and $x\in E$, and let $X$ be a solution to the martingale problem with initial condition $\delta_x$. Since $\langle g,X\rangle$ is a martingale, we have $\langle Bg,X\rangle=0$ and hence $Bg(x)=\langle Bg, X_0\rangle=0$. The form of $Q$ will follow from Lemma~\ref{IIL:form of Q NEW}. To verify its hypotheses, fix $g\in D$ and $\nu\in M_1(E)$, and define $p\in P^D$ by $p(\mu):=-(\langle g,\nu\rangle-\langle g,\mu\rangle)^2$. Then $\partial^2p(\nu)=-2g\otimes g$, $p\le0$, and $p(\nu)=0$, so the positive maximum principle yields
\[
-\langle Q(g\otimes g),\nu^2\rangle = Lp(\nu) \le 0.
\]
Next, fix $g\in D$ and $\nu\in M_1(E)$ such that $g$ is constant on the support of $\nu$. Define $p\in P^D$ by $p(\mu):=\langle g,\mu\rangle^2 - \langle g^2,\mu\rangle$. Then, again, $\partial^2p(\nu)=2g\otimes g$, and Jensen's inequality yields $p\le0$ and $p(\nu)=0$. Consequently,
\[
\langle Q(g\otimes g),\nu^2\rangle = Lp(\nu) \le 0.
\]
The form of $Q$ thus follows from Lemma~\ref{IIL:form of Q NEW}.

To prove the reverse implication, observe that existence of solutions to the martingale problem, along with path continuity, follows from Corollary~\ref{cor2}, as does well--posedness if in addition $\cC $ is bounded. Since $B=0$, it is clear that $\langle g,X\rangle$ is a martingale for every $g\in D$ and every solution $X$ to the martingale problem. This implies that $X$ is a martingale.
\end{proof}

\section{Examples and applications} \label{IIsec:examples}

\subsection{Finite underlying space}\label{unisimpl}

Let $E=\{1,\ldots,d\}$. Then $C_\Delta(E)=C(E)$ is finite-dimensional, so any dense linear subspace must equal the whole space. We therefore take $D=C(E)$. In this setting, any $M_1(E)$-valued process $X$ is of the form $X_t=\sum_{i=1}^d Z^i_t \delta_i$ for some $\Delta^d$-valued process $Z=(Z^1,\ldots,Z^d)$. When $X$ is a polynomial diffusion, Theorem~\ref{IImainthm3} describes its generator $L$ in terms of a kernel $\nu_B$ from $E$ to $E$ and a nonnegative symmetric function $\cC\colon E^2\to\R$. As we now show, the process $Z$ then also solves a martingale problem whose generator can be written down explicitly.

In view of Example~\ref{IIex4}, any polynomial $f$ on $\Delta^d$ can be represented as
\[
f(z) = p(z_1\delta_1 + \cdots + z_d\delta_d)
\]
for some $p\in P^D$. We may then define an operator $A$ acting on such polynomials $f$ by the formula
\[
Af(z) := Lp(z_1\delta_1 + \cdots + z_d\delta_d).
\]
Since $f(Z)=p(X)$ and $Af(Z)=Lp(X)$, it is clear that $Z$ is a solution to the martingale problem for $A$ with polynomials $f$ as test functions. Conversely, if a solution $Z$ to this martingale problem is given, a solution to the martingale problem for $L$ is obtained by setting $X:=\sum_{i=1}^d Z^i \delta_i$.

Next, a computation shows that $A$ has the form
\begin{equation}\label{IIeqn22}
\begin{aligned}
Af(z) &= \sum_{i,j=1}^d \nu_B(i,\{j\}) z_i \bigg(\frac{\partial f}{\partial z_j}(z)-\frac{\partial f}{\partial z_i}(z)\bigg)\\
&\qquad+\frac 1 2 \sum_{i,j=1}^d \cC (i,j)z_iz_j \bigg(\frac{\partial^2f}{\partial z_i^2}(z)+\frac{\partial^2 f}{\partial z_j^2}(z)-2\frac{\partial^2f}{\partial z_i\partial z_j}(z)\bigg).
\end{aligned}
\end{equation}
This can alternatively be written
$
A f(z)={b(z)}^\top\nabla f(z)+\frac 1 2 \tr\big(a(x)\nabla^2 f(x)\big),
$
where the coefficients $b$ and $a$ are given by
\begin{align*}
b_k(z)&:=\sum_{i=1}^{d}\big(\nu_B(i,\{k\})z_i-\nu_B(k,\{i\})z_k\big), && k=1,\ldots,d,\\
a_{k\ell}(z)&:=-\frac 12\cC (k,\ell)z_k z_\ell, && k,\ell=1,\ldots,d,\ k\ne\ell,
\end{align*}
and $a_{kk}(z)=-\sum_{\ell\ne k}a_{k\ell}(z)$. Here well-posedness was obtained by~\citet{FL:16}, which we thus recover as a special case. In particular, $Z$ is a polynomial diffusion on $\Delta^d$ in the sense of~\citet[Definition~2.1]{FL:16}. Furthermore, Theorem~\ref{IImainthm3} yields the moment formula for $X$, which reduces to the corresponding formula for $Z$ given by~\citet[Theorem~3.1]{FL:16}.

\subsection{Underlying space $E\subseteq \R^d$}\label{IIs41}

Let $E\subseteq\R^d$ be a closed subset and set
\[
D := \{f|_E\colon f \in \R+C^\infty_c(\R^d)\}.
\]
Our goal is to analyze Theorem~\ref{IImainthm5} in this setting. If $E$ is not all of $\R^d$, the dynamics of the spatial motion is restricted. Intuitively, its diffusion component must be tangential to the boundary of $E$. This is encoded as follows.
\begin{equation}\label{eqn2}
\Sigma^d(E):=\big\{\tau\in C^1_\Delta(\R^d,\R^{d\times d})\colon \text{$g\in D$, $x\in E$, $g(x)=\max_{E}g$ implies $\tau(x)^\top\nabla g(x)=0$}\big\}.
\end{equation}
Here $C^1_\Delta(\R^d,\R^{d\times d})$ consists of the matrix-valued functions with components in $C^1_\Delta(\R^d)=C_\Delta(\R^d)\cap C^1(\R^d)$.

\begin{lemma}\label{lem1}
Fix $\tau\in \Sigma^d(E)$ with columns $\tau_1,\ldots,\tau_d$. The operators $A_i\colon D\to C_\Delta(E)$ given by 
\begin{equation}\label{eqn1}
A_ig:=\tau_i^\top\nabla g,\quad g\in D,
\end{equation}
satisfy the conditions of Theorem~\ref{IIL:KKT2}. That is, each $A_i$ is the generator of a strongly continuous group of positive isometries of $C_\Delta(E)$, and its domain contains both $D$ and $A_i(D)$.
\end{lemma}

Note that $A_i$ is well-defined by \eqref{eqn1} in the sense that $A_ig$ only depends on $g$ through its values on $E$. This is a direct consequence of the definition~\eqref{eqn2} of $\Sigma^d(E)$.

\begin{proof}
By Proposition 2.5 in \cite{DF:04}, for each $i=1,\ldots,d$, there exists a map $(t,x)\mapsto\phi_i(t,x)$ from $\R\times E$ to $E$ such that
\[
\frac{\partial}{\partial t}\phi_i(t,x) = \tau_i( \phi_i(t,x)), 
\qquad \phi_i(0,x)=x,
\]
and the flow property $\phi_i(s+t,x)=\phi_i(s,\phi_i(t,x))$ holds since $\tau_i \in C^1_\Delta(\R^d,\R^{d})$. This implies that $T_{i,t}g(x) := g(\phi_i(t,x))$, $t\in\R$, defines a strongly continuous group of positive isometries of $C_\Delta(E)$ with generator $A_i$. It is clear that the domain of $A_i$ contains $D$, and it also contains $A_i(D)$ since the components of $\tau_i$ lie in $C^1_\Delta(\R^d)$.
\end{proof}

\begin{theorem}\label{IImainthm}
Let $L\colon P^D\to P$ be a linear operator satisfying \eqref{IIeq:T:Lpol}, where 
\begin{enumerate}
\item 
 $B$ is $E$-conservative and $B1 =0$,
\item \label{it8} $Q$ is given by
$$
Q(g\otimes g)= \cC \Psi (g\otimes g)+\tr\big((\tau^\top \nabla g)\otimes(\tau^\top \nabla g)^\top\big)\qquad g\in D,
$$
where $\tau\in\Sigma^d(E)$ and $\cC \colon E^2\to\R$ is a nonnegative symmetric function,
\item\label{it9} $B-\sum_{i=1}^d (\tau_i^\top \nabla)^2$ satisfies the positive maximum principle on $E$, where $\tau_1,\ldots,\tau_d$ are the columns of $\tau$. 
\end{enumerate}
Then conditions \ref{it7}--\ref{itiii} of Theorem~\ref{IImainthm5} hold.
\end{theorem}
\begin{proof}
This follows directly from Lemma~\ref{lem1}, up to the fact that in \ref{it9} we need to verify that the positive maximum principle holds on $E^\Delta$, not just on $E$. Since $D\subseteq C_c(E)$, this follows from Remark~\ref{IIrem:maxprinciple}\ref{it2}.
\end{proof}

The rest of the section is devoted to the case $d=1$ and $E=\R$. In view of Lemma~\ref{IIlem1}, the operator $B$ should satisfy the positive maximum principle on $E=\R$. It is well-known, see e.g.~\cite{C:65} or \cite{H:98}, that under this condition $B$ is a L\'evy type operator, i.e.
\begin{equation}\label{IIeqn24}
Bg=b g'+\frac 1 2ag''+\int \left( g(\fdot+\xi)-g-\chi(\xi) g' \right) F(\fdot,d\xi),\qquad g\in D,
\end{equation}
for some continuous functions $a$, $b$ with $a\ge0$, a truncation function~$\chi$, and a kernel $F(\fdot,d\xi)$ from $\R$ to $\R$ such that
$\int |\xi|^2\wedge1\,F(\fdot,d\xi)<\infty$. Every operator of this form satisfies $B1=0$ and the positive maximum principle on $\R$.  The following result expresses Theorem \ref{IImainthm} in this setting.

\begin{corollary}\label{IImainthmB}
Let $L\colon P^D\to P$  be a linear operator satisfying \eqref{IIeq:T:Lpol}, where $B$ is given by \eqref{IIeqn24} with $a:=\sigma^2+\tau^2$ for some continuous functions $\sigma$ and $\tau$, and $Q$ is given by
$$Q(g\otimes g)(x,y)=\frac 1 2 \cC (x,y)(g(x)-g(y))^2+\tau(x)\tau(y)g'(x)g'(y),\qquad g\in D,$$
where $\cC \in \widehat C_\Delta(\R^2)$ is  nonnegative and $\tau\in C^1_\Delta(\R)$. Assume also that $B$ is $\R$-conservative. 
Then conditions \ref{it7}--\ref{itiii} of Theorem~\ref{IImainthm5} hold true.
\end{corollary}

The coefficient $\cC$ quantifies the diffusive exchange of mass between different points in the support of $X_t(dx)$. This is perhaps most clearly seen when $E=\{1,\ldots,d\}$; see Section~\ref{unisimpl}. The role of $\tau$ is different, as it governs random fluctuations of the support of $X_t(dx)$. The following example illustrates this point.

\begin{example}
Consider an operator $L$ of the form given in Corollary~\ref{IImainthmB} with $\cC =0$, $Bg=\frac 1 2 g''$, and $\tau=1$ (hence $\sigma=0$). The resulting operator $Q$ is given by $Q(g\otimes g)=g'\otimes g'$. A solution to the martingale problem for $L$ is given by $X=\delta_W$, where $W$ is a Brownian motion. Indeed, applying It\^o's formula to $\langle g,X_t\rangle^k=g(W_t)^k$ for any $g\in D$ and $k\in \N_0$ establishes that \eqref{IIeqnN} is a martingale for any $p\in P^D$.
\end{example}

In this example, as well as in Corollary~\ref{IImainthmB}, a nonzero $\tau$ in the specification of $Q$ is coupled with a corresponding diffusive component in the specification \eqref{IIeqn24} of $B$. The following result shows that this is a general phenomenon.

\begin{proposition}\label{prop1}
Let $L\colon P^D\to P$  be a linear operator satisfying \eqref{IIeq:T:Lpol} with $B$ given by \eqref{IIeqn24}. Suppose that $L$ satisfies the positive maximum principle on $\R$. If $a=0$, then $Q=\cC \Psi$ for some nonnegative symmetric function $\cC :\R^2\to\R$.
\end{proposition}

\begin{proof}
Lemma~\ref{IIlem12} below with $\lambda=0,1,1/2$, along with Lemma~\ref{IIlem1}, imply that the conditions of Lemma~\ref{IIL:form of Q NEW}\ref{IIL:form of Q NEW:1} are satisfied. The result follows.
\end{proof}

The next lemma constitutes the main tool to prove Proposition~\ref{prop1}. 
But it also has other consequences. In particular, it implies that $Q(g\otimes g)(x,y)$ depends on $g$ just through $g(x),g(y), g'(x)$, and $g'(y)$, provided that $L$ satisfies the positive maximum principle on $M_1(E)$. This illustrates that the form of $Q$ as given in Theorem~\ref{IImainthm} is very general.

\begin{lemma}\label{IIlem12}
Let $L\colon P^D\to P$  be a linear operator satisfying \eqref{IIeq:T:Lpol} with $B$ given by \eqref{IIeqn24}. Suppose that $L$ satisfies the positive maximum principle on $\R$.
Then, for all $\lambda\in[0,1]$, $g\in D$, and  $x,y\in \R$ such that $g(x)=g(y)$, we have that
\[
\big\langle Q(g\otimes g),\nu_\lambda^2\big\rangle\leq\big\langle  (a g')^2,\nu_\lambda\big\rangle
,\qquad  \nu_\lambda=\lambda\delta_x+(1-\lambda)\delta_y.
\]
\end{lemma}

 \begin{proof}
 Fix  $g\in D$ such that $g(x)=g(y)$. Since, by Lemma \ref{IIlem1}, $B1=0$ and $Q(g\otimes 1)=0$ it is enough to consider the case $g(x)=g(y)=1$.
The result will follow from Lemma~\ref{IIlem2}. Indeed, if we let $(p_n)_{n\in \N}$ and $(f_n)_{n\in \N}$ be the sequences described there, by the positive maximum principle of $L$ on $\R$ we get
$$
 0\geq Lp_n(\nu_\lambda)= \big\langle  B f_n,\nu_\lambda\big\rangle+\frac 1 2 \Big\langle Q(g\otimes g),\nu_\lambda^2\Big\rangle
$$
and letting $n$ go to $\infty$ we can conclude the proof.

To verify the hypotheses of Lemma~\ref{IIlem2}, observe that Lemma \ref{IIlem1} yields
$$\langle Q(g\otimes g),\nu_\lambda^2\rangle\geq 0\qquad\text{for all $\lambda\in[0,1]$.}$$
Fix some $g\in D$ and $x,y\in \R$ such that $g(z)=g'(z)=0$ for $z\in\{x,y\}$, and suppose that $\|g\|=1$.
Let $F_n:[0,1]\to\R$ be the function defined in Lemma~\ref{IIlem10}.
Consider then the sequence of polynomials given by
$$p_n(\nu)=\langle \tg,\nu\rangle^2F_n\big(\langle H,\nu\rangle\big)-\frac 1n{\langle H,\nu\rangle} ,$$
where, for some compactly supported function $\rho\in C^\infty_\Delta(\R )$ such that $\rho=1$ on some neighborhood of $x$ and $y$ and  $\rho (\R )\subseteq[0,1]$,
$$H(z)={C|z- x|^2|z- y|^2\rho(z)+(1-\rho(z))}.$$
Observe that the conditions  on $g$ guarantee that 
for $C$ big enough $|g|\leq H$ and thus $|\langle \tg,\nu\rangle|\leq \langle H,\nu\rangle$ for all $\nu\in M_1(\R)$. For $\supp(\rho)$ small enough we also have that $\| H\|\leq 1$.
Lemma~\ref{IIlem10} then yields
$
\langle \tg,\nu\rangle^2F_n\big(\langle H,\nu\rangle\big)
\leq \frac 1n {\langle H,\nu\rangle}$ for all $\nu\in M_1(\R),
$
and therefore $p_n\leq0$ on $M_1(\R )$. This automatically implies that $p_n$ has a maximum at $\nu_\lambda$ for all $\lambda\in[0,1]$. Proceeding as in the proof of Theorem~\ref{IImainthm3} we then obtain that $\langle Q(g\otimes g),\nu_\lambda^2\rangle=0$ for any $g\in D$ such that $g(x)=g(y)=1$ and $g'(x)=g'(y)=0$. Choosing $\lambda=0,1,1/2$ we get the result.
\end{proof}

The following example gives a simple condition for well-posedness. We let $G_k$, $C_k$, and  $T_k$ be as in Remark~\ref{IIrem:maxprinciple}\ref{IIrem10}.

\begin{example}\label{ex1}
Consider the setting of Corollary~\ref{IImainthmB}. Suppose that $\sigma^2$ is bounded away from zero, let the jump kernel $F(\fdot,d\xi)$ in \eqref{IIeqn24} be zero, and assume that the parameters $b$ and $\sigma^2$ are Lipschitz continuous and bounded. Then, by Theorem~8.1.6 of \cite{EK:05}, $B$ is $\R$-conservative and the closure of $G_k+T_k$ generates a strongly continuous semigroup on $\widehat C_\Delta(E^k)$ for each $k\in \N$. Since $C_k$ is bounded, $\overline L_k$ generates a strongly continuous contraction semigroup on $\widehat C_\Delta(E^k)$ as well (see e.g.~Theorem~1.7.1 in \cite{EK:05} for more details).  Since Remark~\ref{IIrem:maxprinciple}\ref{IIrem10} shows that  $L_k$ satisfies the positive maximum principle, Remark~\ref{rem2} and Theorem~\ref{IIthm:new} yield the moment formula for all $g\in D^{\otimes k}$. Well-posedness thus follows from Theorem~\ref{IImainthm5}.
\end{example}

\subsection{Conditional laws of jump-diffusions are polynomial}\label{sec:conditionallaw}

In this section we deal with particle systems driven by some idiosyncratic noise (Brownian motion and jumps) and one common Brownian motion. We show that for essentially  \emph{all} such jump diffusions the conditional law with respect to the common Brownian motion is polynomial.

Throughout $E=\R$ and $D\subseteq \R+C_c^\infty(\R)$. Let $b, \sigma, \tau $  and $F(\fdot,d\xi)$ be as in Corollary~\ref{IImainthmB} with the additional integrability conditon $\int |\xi|^2 \wedge |\xi| F(\fdot, d\xi) <\infty$. For these parameters and $\cC=0$ we
define $L$ to be the corresponding polynomial operator as of Corollary~\ref{IImainthmB}. 

Moreover, let $(Z^i)_{i\in\N}$ be a weak solution of the system
\begin{equation}\label{IIIeqn12}
dZ^i_t =b(Z^i_t)dt+\sigma(Z^i_t)dW^i_t +\tau(Z^i_t)dW_t^0+\int \xi\big(\mathfrak{p}^i(dt,d\xi)-F(Z_t^i,d\xi)dt\big), \quad Z^i_0=x\in \R,
\end{equation}
where $W^0$ is a Brownian motion and $(W^1, \mathfrak{p}^1), (W^2, \mathfrak{p}^2)\ldots$ is a sequence of couples of Brownian motions and random measures with compensator $F(\fdot,d\xi)$. We assume that each couple is independent of the other couples and of $W^0$. Note that the generator of each $Z^i$ is given by
$B$ as defined in \eqref{IIeqn24}. 
 
Assume now  that $Z^1, Z^2, \ldots$ are exchangeable and set
$$X_t =\lim_{n\to\infty}\frac 1 n\sum_{i=1}^n\delta_{Z_t^i}.$$
By De Finetti's theorem (see e.g.~Theorem 4.1 in \cite{KK:08} or, for a general overview, also Section 12.3 in \cite{K:13}) we get that $(Z^i_t)_{i\in \N}$ are conditionally i.i.d.~with respect to the invariant $\sigma$-algebra $\Fcal_t^\infty=\sigma(X_s,s\leq t)$ and that $X$ can be expressed as
\begin{align}\label{eq:Xcondlaw}
X_t=\P(Z^1_t \in \fdot |\Fcal_t^\infty).
\end{align}
This implies in particular that for all $g\in D^{\otimes k}$ and $k\in \N$ it holds 
\begin{align}\label{eq:polyXcondlaw}
\langle g, X^k_t\rangle=\E[ g(Z^1_t,\ldots, Z^k_t) |\Fcal^\infty_t].
\end{align}

Note that (see e.g.~Theorem 2.3 in \cite{KX:99}) that under the additional assumption of  pathwise uniqueness for the solution of \eqref{IIIeqn12}, we get that
$$X_t=\P(Z^1_t \in \fdot |\Fcal_t^0),\qquad\text{where}\qquad \Fcal_t^0=\sigma(W^0_s,s\leq t),$$ since $\Fcal_t^0= \Fcal_t^{\infty}$ in this case.

In the following proposition we now show that $X$ is polynomial by proving that it solves the martingale problem for the polynomial operator $L$ specified above.

\begin{proposition}\label{IIIlem4}
Let $X$ be given by \eqref{eq:Xcondlaw}. Then 
$X$ solves the martingale problem for $ L$  with initial condition $\delta_x$.
\end{proposition}

\begin{proof}
Let $g\in D^{\otimes k}$ and set $Z:=(Z^1,\ldots, Z^k)$. Then we get that
$$
N^{g,k}_t:=g(Z_t)-g(x,\ldots,x)-\int_0^t L_kg(Z_s)ds
$$
is a bounded $(\Fcal_t)_{t\geq0}$-martingale, where, in accordance with \eqref{eqn4}, 
$$L_k=kB\otimes\id^{\otimes (k-1)}+\frac{k(k-1)}2\Sigma^\tau\otimes \Sigma^\tau\otimes \id^{\otimes (k-2)}$$
for $\Sigma^\tau g:=\tau g'$.
Since $\Fcal_t^\infty\subseteq\Fcal_t$ this implies that $\E[N^{g,k}_t|\Fcal_t^\infty]$ is an $(\Fcal^\infty_t)_{t\geq0}$-martingale and hence setting $p(\nu):=\langle g,\nu^k\rangle$ we can compute using \eqref{eq:polyXcondlaw}
\begin{align*}
\E[p(X_t)|\Fcal_s^\infty]-p(X_s)&=\E\bigg[\int_s^tL_kg( Z_u)du\bigg|\Fcal^\infty_s\bigg]\\
&=\E\bigg[\int_s^t\E[L_kg( Z_u)|\Fcal^\infty_u]du\bigg|\Fcal^\infty_s\bigg]
=\E\bigg[\int_s^t Lp(X_u)du\bigg|\Fcal^\infty_s\bigg]
\end{align*}
proving that $X$ is a solution to the martingale problem for $ L$.
\end{proof}

\begin{appendices}

\section{Proof of Theorem \ref{IIT:Lpol} and a generalization}\label{IIE}

We first prove Theorem \ref{IIT:Lpol}. Assume first $L$ is of the stated form. Then for monomials $p(\nu)=\langle g,\nu\rangle^k$ with $g\in D$, $k\in\N$ and $\nu\in M_1(E)$ one has
\begin{align*}
Lp(\nu) &=  \big\langle B\big(\partial p(\nu)\big), \nu\big\rangle + \frac{1}{2} \big\langle Q\big(\partial^2p(\nu)\big), \nu^2\big\rangle \\
&=  k\langle g,\nu\rangle^{k-1} \langle Bg, \nu\rangle + \frac{1}{2}k(k-1)\langle g,\nu\rangle^{k-2}  \langle Q(g\otimes g), \nu^2\rangle,
\end{align*}
which is a polynomial in $\nu$ of degree at most $k$. Moreover, $L1=0$. By linearity, this shows that $L$ is $M_1(E)$-polynomial. Next, a direct calculation yields
\[
\Gamma(p,q)(\nu) = \Big\langle Q\big(\partial p(\nu)\otimes \partial q(\nu)\big), \nu^2\Big\rangle\qquad\text{for all $\nu\in M_1(E)$,}
\]
which is easily seen to be an $M_1(E)$-derivation due to the product rule give in Lemma~\ref{IIL:Psmooth}\ref{IIL:Psmooth:3}.

Conversely, assume $L$ is $M_1(E)$-polynomial and $\Gamma$ is an $M_1(E)$-derivation. Consider arbitrary first degree monomials $q(\nu)=\langle g,\nu\rangle$ and $r(\nu)=\langle h,\nu\rangle$, $g,h\in D$. The $M_1(E)$-polynomial property and Corollary \ref{IIrem5} yield
\[
Lq(\nu) = \langle Bg, \nu\rangle\qquad\text{for all}\ \nu\in M_1(E),
\]
for some map $B:D\to C_\Delta(E)$ that are easily seen to be linear due to the linearity of $L$. Furthermore, the $M_1(E)$-polynomial property, definition \eqref{IIeq:Gamma} of $\Gamma$, and Corollary \ref{IIrem5} imply that 
\[
\Gamma(q,r)(\nu) =  \langle Q(g\otimes h),\nu^2\rangle\qquad\text{for all}\ \nu\in M_1(E),
\]
where $Q$ inherits symmetry and linearity from $\Gamma$ and take values in $\widehat C_\Delta(E^2)$. Thus, by taking linear combinations, we can and do extend them to operators on $D\otimes D$.

Explicit calculation now shows that $Lp$ is of the form~\eqref{IIeq:T:Lpol} for $p=q$ and $p=q^2$. Furthermore, since $\Gamma$ is an $M_1(E)$-derivation we have $\Gamma(1,1)=2\Gamma(1,1)$, hence $\Gamma(1,1)=0$, and therefore $L1=L(1^2)=0+2L1$. Thus $L1=0$, so that~\eqref{IIeq:T:Lpol} holds also for $p=1$.

We now make more substantial use of the fact that $\Gamma$ is an $M_1(E)$-derivation in order to extend~\eqref{IIeq:T:Lpol} to higher degree monomials. We proceed by induction on $k$, and assume $Lp$ is of the form~\eqref{IIeq:T:Lpol} for all $p=q^l$, $l\le k$. So far we have proved this for $k=2$. The definition~\eqref{IIeq:Gamma} of $\Gamma$ and the fact that it is an $M_1(E)$-derivation give the identity on $M_1(E)$
\[
L(q^{k+1}) = 2qL(q^k) - q^2 L(q^{k-1})+ q^{k-1}\Gamma(q,q)
\]
for $k\ge 2$. Due to the induction assumption, the right-hand side can be computed explicitly using~\eqref{IIeq:T:Lpol}. The result is
\[
(k+1)q(\nu)^k \langle Bg,\nu\rangle + \frac{1}{2} (k+1)k q(\nu)^{k-1} \langle Q(g\otimes g),\nu^2\rangle,
\]
which is equal to $\langle B(\partial p(\nu)), \nu\rangle + \frac{1}{2} \langle Q(\partial^2p(\nu)), \nu^2\rangle$ with $p=q^{k+1}$, for all $\nu\in M_1(E)$. This concludes the induction step. It follows by induction that \eqref{IIeq:T:Lpol} holds for all  monomials $\langle g,\nu\rangle^k$, and by linearity for all $p\in P^D$.
Finally, the uniqueness assertion is immediate from the way $B$ and $Q$ were obtained above. This completes the proof of Theorem \ref{IIT:Lpol}.\qed

We now state a generalization of Theorem \ref{IIT:Lpol}, where $M_1(E)$ is replaced by a general state space. We let $E$ be a locally compact Polish space, $D \subseteq C_\Delta(E)$ be a dense linear subspace, and fix $\s\subseteq M(E)$.

\begin{theorem}\label{IIT:Lpol gen}
Let $L:P^D\to P$ be a linear operator. Then $L$ is $\s$-polynomial and its carr\'e-du-champs operator $\Gamma$ is an $M_1(E)$-derivation if and only if $L$ admits a representation
\begin{align*}
Lp(\nu) = &B_0(\partial p(\nu)) + \Big\langle B_1(\partial p(\nu)),\nu\Big\rangle\\
& + \frac{1}{2} \bigg(Q_0(\partial^2 p(\nu)) + \Big\langle Q_1(\partial^2 p(\nu)), \nu\Big\rangle + \Big\langle Q_2(\partial^2 p(\nu)), \nu^2\Big\rangle\bigg), \qquad \nu\in \s
\end{align*}
for some linear operators $B_0:D\to\R$, $B_1:D\to  C_\Delta(E)$, $Q_0:D\otimes D\to\R$, $Q_1:D\otimes D\to C_\Delta(E)$, $Q_2:D\otimes D\to \widehat C_\Delta(E^2)$. If $\s$ contains an open subset of $M(E)$, these operators are uniquely determined by $L$.
\end{theorem}
\begin{proof}
The proof of this result follows the proof of Theorem \ref{IIT:Lpol}.
\end{proof}

\section{Proof of Theorem~\ref{IImainthm3}}\label{appB}

Assume $L$ satisfies~\eqref{IIeq:T:Lpol} with $B$ and $Q$ as in \eqref{IIeqn1_2}, where $\nu_B$ is a nonnegative, finite kernel from $E$ to $E$, and $\cC \colon(E^\Delta)^2\to\R$ is nonnegative, symmetric, bounded, and continuous on $(E^\Delta)^2\setminus\{x=y\}$. Clearly $Q$ is bounded with operator norm $2\|\cC \|$. Identifying $C_\Delta(E)$ and $C(E^\Delta)$, we infer from Lemma~\ref{IIML_L1} that $B$ is bounded, satisfies $B1=0$ as well as the positive maximum principle on $E^\Delta$, and that $\{e^{tB}\}_{t\ge0}$ is a strongly continuous contraction semigroup. By considering any sequence of functions $g_n\in C_0(E)$ with $0\le g_n(x)\uparrow 1$ for all $x\in E$, and using that $\nu_B(x,\{\Delta\})=0$ for all $x\in E$, one sees that $B$ is $E$-conservative. Theorem~\ref{IImainthm5} then yields that $L$ is $M_1(E)$-polynomial and its martingale problem has an solution with continuous paths for every initial condition $\nu\in M_1(E)$. Well--posedness follows by Lemma~\ref{IIlem3}.

We now prove the opposite implication. Assume $L$ is $M_1(E)$-polynomial, its martingale problem is well--posed, and all solutions have continuous paths. Theorem~\ref{IIT:Lpol} and Lemma~\ref{IIlem5} imply that $L$ satisfies~\eqref{IIeq:T:Lpol}, and then also the positive maximum principle on $M_1(E)$ due to Lemma~\ref{IIIlem9}.

By Lemma \ref{IIlem1} $B$ satisfies the positive maximum principle on $E$ and Lemma~\ref{IIML_L1} thus shows that $B$ has the form in \eqref{IIeqn1_2} for some nonnegative, finite kernel $\nu_B$ from $E^\Delta$ to $E^\Delta$. Additionally, $B$ is bounded, satisfies the positive maximum principle on $E^\Delta$, and is the generator of the strongly continuous contraction semigroup $\{e^{tB}\}_{t\ge0}$.
We must prove that $\nu_B(x,\{\Delta\})=0$ for all $x\in E$; this will allow us to view $\nu_B$ as a kernel from $E$ to $E$.

Assume by contradiction that there exists some $x \in E$ such that $\nu_B(x,\{\Delta\})>0$. Let $Z$ be the Markov process associated to the semigroup $\{e^{tB}\}_{t\ge0}$. Then, by approximating $1_{\{\cdot \in \Delta\}}$ by a sequence of bounded continuous functions $g^n$ and applying relation \eqref{IIeq:momentformulaZ}, we find
\begin{align*}
0 &< \mathbb{P} [ Z_t \in \Delta | Z_0=x]= \mathbb{E}[1_{\{Z_t \in \Delta\}}|  Z_0=x]\\ &=\lim_{n \to \infty} \mathbb{E}[g^n(Z_t) |  Z_0=x]= \lim_{n \to \infty} \mathbb{E}[\langle g^n(\cdot), X_t \rangle | X_0=\delta_x]=
\mathbb{E}[\langle 1_{\{\cdot \in \Delta\}}, X_t \rangle | X_0=\delta_x]
\end{align*}
for all $t \geq 0$. This contradicts the fact that $X_t$ is $M_1(E)$-valued and proves that $B$ is of the stated form.

The form of $Q$ will follow from Lemma~\ref{IIL:form of Q NEW}. To verify its hypotheses, note that by Lemma \ref{IIlem1} $\langle Q(g\otimes g),\nu^2\rangle\geq 0$. Next, fix some $g\in D$ and $\nu\in M_1(E)$ such that $g=0$ on the support of $\nu$, and suppose that $\|g\|=1$. For each $n\in\N$, define the polynomial
\[
p_n(\mu) = \langle g,\mu\rangle^2 F_n\left( \langle |g|,\mu\rangle\right) - \frac{1}{n}\langle |g|,\mu\rangle,
\]
where $F_n$ is as in Lemma~\ref{IIlem10}. Since $D=C_\Delta(E)$, we have $p_n\in P^D$. Moreover, since $F_n(z)zn\le1$ for all $z\in[0,1]$, we get
\[
\langle g,\mu\rangle^2 F_n\left( \langle |g|,\mu\rangle \right) \le    \frac{1}{n} \langle |g|,\mu\rangle, \quad \mu\in M_1(E),
\]
and therefore $p_n\le0$ on $M_1(E)$. Since $g=0$ on the support of $\nu$, $p_n(\nu)=0$. Applying the positive maximum principle and using the form \eqref{IIeq:T:Lpol} of $L$, as well as  $\langle g,\nu\rangle=\langle |g|,\nu\rangle=0$ and $F_n(0)=1$  we obtain
\[
0 \ge Lp_n(\nu) = -\frac1n \langle B(|g|),\nu\rangle + \langle Q(g\otimes g),\nu^2\rangle
\]
for all $n$, whence $\langle Q(g\otimes g),\nu^2\rangle\le0$. By scaling, this actually holds for any $g\in D$ and $\nu\in M_1(E)$ such that $g=0$ on the support of $\nu$. If $g$ equals some other constant $c\in\R$ on the support of $\nu$, we still get 
\[
\langle Q(g\otimes g),\nu^2\rangle=\langle Q((g-c)\otimes (g-c)),\nu^2\rangle\le0
\]
using that $Q(g\otimes1)=0$ by Lemma~\ref{IIlem1}. Thus Lemma~\ref{IIL:form of Q NEW}\ref{IIL:form of Q NEW:2} holds, and we conclude that $Q = \cC  \Psi$ for some nonnegative symmetric function $\cC \colon E^2\to\R$. It remains to use that $\cC \Psi(g) \in \widehat{C}_{\Delta}(E^2)$ to show that this function can be extended to a bounded continuous function on $(E^\Delta)^2\setminus\{x=y\}$. 

Continuity is clear. For proving boundedness, choose a sequence of pairs $(x_n,y_n)\in (E^\Delta)^2\setminus \{x=y\}$ such that $\cC (x_n,y_n)\xrightarrow{n\to\infty} \infty$. Since we can assume without loss of generalities that $\cC (x_i,y_i)>0$, $x_i\neq x_j$, $x_i\neq y_j$, and $y_i\neq y_j$ for all $i,j\in \N$, we can construct $g\in C_\Delta(E)$ such that 
$$(g(x_n)-g(y_n))^4=\cC (x_n,y_n)^{-1}.$$
This yields $\cC (x_n,y_n)\Psi(g\otimes g)(x_n,y_n)=\cC (x_n,y_n)^{1/2}$ proving that $\cC \Psi(g\otimes g)$ is unbounded and providing the necessary contradiction.
\qed

\begin{lemma}\label{IIlem10}
Define $F_n(z):=\frac{n-1}n (1-z)^n+\frac 1 n$ for all $z\in[0,1]$. Then 
$$F_n(z)\in[0,1],\qquad F_n(z)zn\leq1,\qquad\text{and}\qquad F_n(z)\sqrt{zn}\leq1,$$
for all $z\in [0,1]$.
\end{lemma}

\section{Auxiliary lemmas}

Let $E$ be a locally compact Polish space.

\begin{lemma}\label{IIlem1}
Let $D\subseteq C_\Delta(E)$ be a dense linear subspace containing the constant function~$1$, and let $L\colon P^D\to P$ be a linear operator satisfying~\eqref{IIeq:T:Lpol} and the positive maximum principle on $M_1(E)$. Then $B$ satisfies the positive maximum principle on $E$, $B1=0$, $\langle Q(g\otimes g),\nu^2\rangle\geq 0$, and $Q(g\otimes 1)=0$ for all $g\in D$ and $\nu\in M_1(E)$.
\end{lemma}

\begin{proof}
By \eqref{IIeq:T:Lpol} we get $L1=0$. Note also that
for any $g\in D$ and $x\in E$ such that $g(x)=\max_Eg\ge0$, the polynomial $p(\nu)=\langle g,\nu\rangle$ lies in $P^D$ and satisfies $p(\delta_x)=\max_{M_1(E)}p\ge0$. Thus $Bg(x)=Lp(\delta_x)\le0$. Furthermore, taking $p(\nu)=\langle 1,\nu\rangle$ we get $p\equiv1$ on $M_1(E)$ and hence $B1(x)=Lp(\delta_x)=0$ for all $x\in E$.
Fix then $g$ and $\nu$ as in the lemma and define $p\in P^D$ by $p(\mu)=-(\langle g,\nu\rangle-\langle g,\mu\rangle)^2$. Then $p\le0$, $p(\nu)=0$, $\partial p(\nu)=0$, and $\partial^2p(\nu)=-2g\otimes g$, so the positive maximum principle yields
$
-\langle Q(g\otimes g),\nu^2\rangle = Lp(\nu) \le 0.
$
Furthermore, taking $p(\nu)=\langle g\otimes 1,\nu^2\rangle-\langle g,\nu\rangle$ we get $p\equiv 0$ on $M_1(E)$ and hence 
$0=\langle g,\nu\rangle\langle B1,\nu\rangle+ \langle Q(g\otimes 1),\nu^2\rangle= \langle Q(g\otimes 1),\nu^2\rangle$
 for all $\nu\in M_1(E)$, proving the claim.

\end{proof}

\begin{lemma} \label{IIML_L1}
Let $B\colon C(E^\Delta)\to C(E^\Delta)$ be a linear operator. Then $B1=0$ and $B$ satisfies the positive maximum principle on $E$ if and only if there is a nonnegative, finite kernel $\nu_B$ from $E$ to $E^\Delta$ such that
\begin{equation} \label{IIeqML_L1}
Bg(x) = \int ( g(\xi)-g(x))\nu_B(x,d\xi)
\end{equation}
for all $x\in E$ and $g\in C(E^\Delta)$. In this case, $B$ is bounded and satisfies the positive maximum principle on $E^\Delta$, and $\{e^{tB}\}_{t\ge0}$ is a strongly continuous contraction semigroup. Moreover, there is some nonnegative (finite) measure $\nu_B(\Delta,\fdot)$ such that \eqref{IIeqML_L1} holds also for $x=\Delta$.
\end{lemma}

\begin{proof}
Assume there is a nonnegative, finite kernel $\nu_B$ from $E$ to $E^\Delta$ such that \eqref{IIeqML_L1} holds for all $x\in E$ and $g\in C(E^\Delta)$. Then clearly $B1=0$. Suppose $g\in C(E^\Delta)$, $x\in E$, and $g(x)=\max_E g\ge0$. Then $g(x)=\max_{E^\Delta}g$, so that $g(\xi)-g(x)\le0$ for all $\xi\in E^\Delta$ and hence $Bg(x)\le0$. Thus $B$ satisfies the positive maximum principle on $E$, which proves sufficiency.

To prove necessity, assume $B1=0$ and $B$ satisfies the positive maximum principle on $E$. By Lemmas~4.2.1 and~1.2.11 in \cite{EK:05}, the restriction $B|_{C_0(E)}$ is dissipative, hence closable, and even closed since it is globally defined on $C_0(E)$. By the closed graph theorem $B|_{C_0(E)}$ is bounded, and then so is $B$ since $B1=0$. Pick any $g\in C(E^\Delta)$ with $g(\Delta)=\max_{E^\Delta} g\ge0$. Then $g-g(\Delta)\le0$, so there exist functions $h_n\in C_c(E)$ with $h_n\le0$ and $h_n\to g-g(\Delta)$ uniformly. Then $Bh_n\to B(g-g(\Delta))=Bg$ uniformly as well. Taking $x_n$ such that $h_n(x_n)=0$ and $x_n\to\Delta$, we obtain $Bg(\Delta)=\lim_{n\to\infty}Bh_n(x_n)\le0$. We have thus proved that $B$ is bounded and satisfies the positive maximum principle on $E^\Delta$. As a result, Lemma~4.2.1 and Theorem~1.7.1 in \cite{EK:05} yield that $\{e^{tB}\}_{t\ge0}$ is a strongly continuous contraction semigroup.

It remains to exhibit a kernel $\nu_B$ from $E^\Delta$ to $E^\Delta$ such that \eqref{IIeqML_L1} holds for all $x\in E^\Delta$ and $g\in C(E^\Delta)$. To this end, fix $x\in E^\Delta$ and define $h\in C(E^\Delta)$ by $h(y):=d(x,y)$, where $d(\fdot,\fdot)$ is a compatible metric for the Polish space $E^\Delta$. Since $B$ satisfies the positive maximum principle on $E^\Delta$, the map
\[
C(E^\Delta)\to\R, \qquad g\mapsto B(g h)(x)
\]
is a positive linear functional. By the Riesz--Markov representation theorem, there is a measure $\mu(x,\fdot)\in M_+(E^\Delta)$ such that $B(g h)(x)=\int_{E^\Delta}g(\xi)\mu(x,d\xi)$ for all $g\in C(E^\Delta)$. We define
\[
\nu_B(x,d\xi) :=  1_{E^\Delta\setminus\{x\}}(\xi)\frac{1}{ h(\xi)}\mu(x,d\xi),
\]
which is permissible since $h(y)>0$ for all $y\ne x$. For every $g\in C_c(E^\Delta\setminus\{x\})$ we have $g/ h \in C(E^\Delta)$, and therefore
\[
Bg(x) = B\left(\frac{g}{ h}\, h\right)(x) = \int_{E^\Delta}\frac{g(\xi)}{ h(\xi)}\mu(x,d\xi) = \int_{E^\Delta}g(\xi)\nu_B(x,d\xi).
\]
Since $B$ is bounded, the identity $Bg(x)=\int_{E^\Delta}g(\xi)\nu_B(x,d\xi)$ extends by continuity to all $g\in C(E^\Delta)$ with $g(x)=0$. Thus, using also that $B1=0$,
\[
Bg(x)=B(g-g(x))(x) = \int_{E^\Delta}(g(\xi)-g(x))\nu_B(x,d\xi).
\]
Repeating this for every $x\in E^\Delta$ yields that $\nu_B$ satisfies \eqref{IIeqML_L1} for all $x\in E^\Delta$ and $g\in C(E^\Delta)$. To see that $\nu_B(x,E^\Delta)<\infty$, just note that $\int_{E^\Delta}g(\xi)\nu_B(x,d\xi)\le\|B\|$ whenever $g\in C(E^\Delta)$ satisfies $0\le g\le 1$ and $g(x)=0$. Measurability of $\nu_B(\fdot,A)$ for every Borel set $A\subseteq E^\Delta$ follows from a monotone class argument, so that $\nu_B$ is indeed a kernel from $E^\Delta$ to $E^\Delta$.
\end{proof}

\begin{lemma} \label{IIL:form of Q NEW}
Let $D\subseteq C_\Delta(E)$ be a dense linear subspace containing the constant function~$1$, and let $Q\colon D\otimes D\to \widehat C_\Delta(E^2)$ be a linear operator. The following conditions are equivalent:
\begin{enumerate}
\item\label{IIL:form of Q NEW:1} $Q(g\otimes g)(x,y)\ge0$ for all $g\in D$ and $x,y\in E$, with equality if $g(x)=g(y)$.
\item\label{IIL:form of Q NEW:2} $\langle Q(g\otimes g),\nu^2\rangle\ge0$ for all $g\in D$ and $\nu\in M_1(E)$, with equality if $g$ is constant on the support of $\nu$.
\end{enumerate}
If either condition is satisfied, then $Q$ is of the form $Q=\cC \Psi$ for some nonnegative symmetric function $\cC \colon E^2\to\R$.
\end{lemma}

\begin{proof}
It is clear that \ref{IIL:form of Q NEW:1} implies \ref{IIL:form of Q NEW:2}. For the converse, first note that for any $x\in E$ and $g\in D$, trivially $g$ is constant on the support of $\delta_x$. Thus $Q(g\otimes g)(x,x)=\langle Q(g\otimes g),\delta_x^2\rangle=0$. Taking $\nu=\frac12(\delta_x+\delta_y)$ for any $x,y\in E$ then yields $Q(g\otimes g)(x,y) = \langle Q(g\otimes g),\nu^2\rangle \ge0$, with equality if $g(x)=g(y)$ since $g$ is then constant on the support of $\nu$. This proves that \ref{IIL:form of Q NEW:2} implies \ref{IIL:form of Q NEW:1}.

It remains to obtain the stated form of $Q$ under the assumption that \ref{IIL:form of Q NEW:1} holds. If $E$ is a singleton then $Q=0$, so we may assume that $E$ contains at least two points. Fix $x,y\in E$ with $x\ne y$. Due to \ref{IIL:form of Q NEW:1}, the map $(g, h)\mapsto Q(g\otimes h)(x,y)$ is bilinear and positive semidefinite, and therefore satisfies the Cauchy--Schwarz inequality
\[
|Q(g\otimes h)(x,y)| \le \sqrt{ Q(g\otimes g)(x,y)}\, \sqrt{ Q(h\otimes h)(x,y)}.
\]
Along with \ref{IIL:form of Q NEW:1} this implies that $Q(g\otimes h)(x,y)$ depends on $g$ and $h$ only through their values at $x$ and $y$. Moreover, since $D$ is dense in $C_\Delta(E)$, for every $a\in\R^2$ there exists $g\in D$ such that $a=(g(x),g(y))$. Thus there is a unique map $T\colon\R^2\times\R^2\to\R$ such that
\[
Q(g\otimes h)(x,y) = T(a,b) \quad\text{where}\quad a=\begin{pmatrix}g(x)\\ g(y)\end{pmatrix},\ b=\begin{pmatrix}h(x)\\ h(y)\end{pmatrix}.
\]
The map $T$ inherits bilinearity and positive semidefiniteness. Since $Q(g\otimes 1)(x,y)=0$ due to the Cauchy--Schwarz inequality and \ref{IIL:form of Q NEW:1}, we also have $T(a,b)=0$ for $b=(1,1)$. This implies that $T(a,b)=\frac12 \cC (x,y)(a_1-a_2)(b_1-b_2)$ for some $\cC (x,y)\in\R_+$. Thus,
\[
Q(g\otimes h)(x,y) = \frac12 \cC (x,y) (g(x)-g(y))(h(x)-h(y)) = \cC (x,y)\Psi(g\otimes h)(x,y).
\]
Defining $\cC (x,x)$ arbitrarily, we obtain the map $\cC \colon E^2\to\R$, which is symmetric due to the symmetry of $Q(g\otimes h)$.
\end{proof}

Consider now the setting of Lemma~\ref{IIlem12}, i.e.~$E=\R$ and
$D=  \R+C^\infty_c(\R).$
\begin{lemma}\label{IIlem2}

Consider two operators $B:D\to C_\Delta(\R)$ and $Q:D\otimes D\to \widehat C_\Delta(\R^2)$ such that $B$ is as in \eqref{IIeqn24} and $Q$ satisfies
\begin{align*}
Q(h\otimes h)(x,y)\ge0&\text{ for all }h\in D,\\
&\qquad\text{ with equality if }h(x)=h(y)\text{ and }h'(x)=h'(y)=0.
\end{align*}
Then, for each $g\in D$ and  $x,y\in \R$ such that $g(x)=g(y)=1$ there exists a sequence $(p_n)_{n\in\N}\subseteq P^D$ such that 
$$p_n(\nu_\lambda)=\max_{M_1(\R )}p_n,\quad \partial p_n(\nu_\lambda)=f_n,\quad\text{and}\quad \big\langle Q(\partial^2 p_n(\nu_\lambda)),\nu_\lambda^2\big\rangle=\big\langle Q(g\otimes g),\nu_\lambda^2\big\rangle$$
for all $n\in\N$ and $\lambda\in [0,1]$, where $\nu_\lambda=\lambda\delta_x+(1-\lambda)\delta_y$ and $(f_n)_{n\in\N}$ satisfies
$$\lim_{n\to\infty}-2 B f_n( z)=(a( z)g'( z))^2,\qquad z\in\{ x, y\}.$$
\end{lemma}
\begin{proof}
Fix  $g\in D$ such that $g(x)=g(y)=1$.
Let $F_n:[0,1]\to\R$ as in Lemma~\ref{IIlem10} and fix a compactly supported function $\rho\in C_c^\infty(\R)$ such that $\rho=1$ on some neighborhoods of $x$ and $ y$ and  $\rho (\R )\subseteq[0,1]$. Set then 
$$\overline g_n( z)=1+g'( x)( z- x)\frac{(z- y)^2}{( x- y)^2}F_{n^4}\bigg(\frac{| z- x|^2}{C_ x}\bigg)
+g'( y)( z- y)\frac{( z- x)^2}{( x- y)^2}F_{n^4}\bigg(\frac{| z- y|^2}{C_ y}\bigg),$$
 where $C_ x=2\sup_{ z\in \supp(\rho)}(z- x)^2$.
Setting $g_n=1+(\overline g_n -1)\rho$ we get $g_n\in \R+C_c^\infty(\R)= D$.
For $n$ even, define now the polynomial
$$p_n(\nu)=\frac 1 {n(n-1)}\big(\langle g_n,\nu\rangle^n-\langle g_n^n,\nu\rangle\big).$$
Since $p_n(\nu_\lambda)=0$ and by Jensen inequality $p_n\leq0$, we can conclude that $\nu_\lambda$ maximizes $p_n$ for all $n$ even and $\lambda\in[0,1]$.  Observe that 
$$
 \partial p_n(\nu_\lambda)=\frac 1 {n-1}\Big(g_n-\frac 1 n g_n^n\Big)=:f_n\quad\text{ and}\quad\partial^2 p_n(\nu_\lambda)=g_n\otimes g_n.
$$
Proceeding as in the proof of Lemma \ref{IIL:form of Q NEW}, we can use the assumptions on $Q$ to prove that $Q(g\otimes h)(x,y)$ depends on $g$ and $h$ only through their values and the values of their derivatives at $x$ and $y$. 
Since $g_n( z)=g( z)=1$ and $g'_n( z)=g'( z)$ for all $n$ even and $ z\in\{ x, y\}$, this implies that
$ \big\langle Q(g_n\otimes g_n),\nu_\lambda^2\big\rangle= \big\langle Q(g\otimes g),\nu_\lambda^2\big\rangle.$
Finally, the representation of $B$ given by $\eqref{IIeqn24}$ yields
$$
-2 B f_n( z)= \big(a( z)g'( z)\big)^2
 -2\int \frac 1 {n-1}\bigg(g_n( z+\xi)-\frac 1 n g_n( z+\xi)^n\bigg)-\frac 1 n\ F( z,d\xi),
$$
for all $z\in \{x,y\}$. Since by the dominated convergence theorem the integral term converges to 0 for $n$ going to $\infty$, this concludes the proof.
\end{proof}

\section{Existence for martingale problems} \label{app_existence}

The purpose of this section is to establish the (essential) equivalence
between the existence of a solution to the martingale problem for $L$ and the positive maximum principle for $L$.  

Here, $E$ is a locally compact Polish space, $D$ a dense linear subspace of $C_\Delta (E)$ containing the constant function $1$, and $L\colon P^D\to P$ a linear operator satisfying \eqref{IIeq:T:Lpol}.

The first lemma asserts that the positive maximum principle is implied if a solution to the martingale problem exists.

\begin{lemma}\label{IIIlem9}
If there exists a solution $X$ to the martingale problem for $L$ for each initial condition in $M_1(E)$, then $L$ satisfies the positive maximum principle on $M_1(E)$.
\end{lemma}
The proof of Lemma~\ref{IIIlem9} is standard and we thus omit it. See for instance the proof of Lemma~2.3 in \cite{FL:16}.

The next lemma is an adaptation of a classical result from \cite{EK:05}. For the application of this result it is crucial that $L$ is an operator on the space of bounded continuous functions on a locally compact, separable, metrizable space. Since this is not  the case for $M_1(E)$ if $E$ is noncompact, we work on $M_1(E^\Delta)$, which is a compact Polish space with respect to the topology of weak convergence.

The result of \cite{EK:05} can then be applied and we just have to check  that if the initial condition of an $M_1(E^\Delta)$ solution $X$ assigns mass 1 to $E$, then $X_t(E)=1$ almost surely for each $t\geq0$, so that the solution actually takes values in $M_1(E)$.

\begin{lemma}\label{IIIlem8}
Suppose that $L$ satisfies the positive maximum principle on $M_1(E^\Delta)$. If $B$ is $E$-conservative, then there exists a solution to the martingale problem for  $L$ for every initial condition in $M_1(E)$. 
\end{lemma}

\begin{proof}
Recall that because of Lemma~\ref{IIL:Psmooth}\ref{IIL:Psmooth:c1}, the operator $L$ can be seen as an operator on the space polynomials on $M_1(E^\Delta)$.
The first part of the proof consists then in proving that if $L$ satisfies the positive maximum principle on $M_1(E^\Delta)$ then there exists an $M_1(E^\Delta)$-valued solution to the martingale problem for  $L$ for every initial condition in $M_1(E^\Delta)$.
This result is a consequence of Theorem~4.5.4  in \cite{EK:05} and the successive Remark~4.5.5. We now explain how the necessary conditions hold true. Observe that $M_1(E^\Delta)$ is a compact separable metrizable space and, by Lemma~\ref{IIIlem1}, that
\[
P^D(M_1(E^\Delta)):=\{p|_{M_1(E^\Delta)}\colon p\in P^D\}
\]
 is a dense subset of the space of continuous functions on $M_1(E^\Delta)$. Moreover, the positive maximum principle implies that $Lp|_{M_1(E^\Delta)}=Lq|_{M_1(E^\Delta)}$ for all $p,q\in P^D$ such that $p|_{M_1(E^\Delta)}=q|_{M_1(E^\Delta)}$. We may thus regard $L$ as an operator on the space of continuous functions on $M_1(E^\Delta)$ with domain $P^D(M_1(E^\Delta))$. 

For the second part, recall that by definition of $E$-conservativity  there exist functions 
 $g_n\in D\cap C_0(E)$ such that  $\lim_{n\to\infty}g_n=1$, and $\lim_{n\to\infty}(B g_n)^-=0$ bounded pointwise
on $E$ and $E^\Delta$, respectively. By the dominated convergence theorem, \eqref{IIeqnN}, and  Fatou's lemma we can compute
$$\E[X_t(E)]=\lim_{n\to \infty}\E[\langle g_n,X_t\rangle]
=\lim_{n\to \infty}\bigg(\langle g_n,\nu\rangle+\E\bigg[\int_0^t \langle Bg_n,X_s\rangle ds\bigg]\bigg)
\geq \nu(E)=1.$$
Finally, note that a c\`adl\`ag process $X$ on $M_1(E^\Delta)$ such that $X_t(E)=1$ almost sure is c\`adl\`ag also with respect to the topology of weak convergence on $M_1(E)$.
\end{proof}

\end{appendices}


\end{document}